\documentclass[11pt,oneside]{amsart}
\usepackage{amsfonts}
\usepackage{amsmath}
\usepackage{amssymb}
\usepackage{mathrsfs}
\usepackage{geometry}
\usepackage{graphicx}
\usepackage{subfigure}
\usepackage{cite}
\usepackage{hyperref}
\usepackage{microtype}
\usepackage{enumerate}
\usepackage{mathrsfs}
\usepackage{tikz}
\usepackage{tikz-cd}
\usepackage{color}
\usepackage{ragged2e}
\allowdisplaybreaks[4]

\numberwithin{equation}{section}

\newcommand*{\circled}[1]{\lower.7ex\hbox{\tikz\draw (0pt, 0pt)%
    circle (.5em) node {\makebox[1em][c]{\small #1}};}}

\newcommand{\al}{\alpha}

\newcommand{\ga}{\gamma}
\newcommand{\Ga}{\Gamma}

\newcommand{\ve}{\varepsilon}

\newcommand{\vp}{\varphi}

\newcommand{\R}{\mathbb{R}}
\newcommand{\N}{\mathbb{N}}

\newcommand{\Z}{\mathbb{Z}}
\newcommand{\T}{\mathbb{T}}

\newcommand{\f}{\forall}

\newcommand{\ccc}{\cdot\cdot\cdot}

\newcommand{\n}[1]{\Vert #1\Vert }

\newcommand{\bbn}[1]{\Big\Vert #1 \Big \Vert }
\newcommand{\lr}[1]{\left\{ #1\right\} }
\newcommand{\lrc}[1]{\left[ #1\right] }
\newcommand{\lrs}[1]{\left( #1\right) }

\newcommand{\lra}[1]{\langle #1\rangle}

\newcommand{\wt}[1]{\widetilde{#1} }

\newcommand{\pa}{\partial}

\newcommand{\cf}{{\mathcal F}}

\begin{document}
\newtheorem{theorem}{Theorem}[section]
\newtheorem{lemma}[theorem]{Lemma}
\theoremstyle{definition}
\newtheorem{definition}[theorem]{Definition}
\newtheorem{example}[theorem]{Example}
\newtheorem{remark}[theorem]{Remark}

\numberwithin{equation}{section}

\newtheorem{proposition}[theorem]{Proposition}
\newtheorem{corollary}[theorem]{Corollary}
\newtheorem{goal}[theorem]{Goal}
\newtheorem{algorithm}{Algorithm}

\renewcommand{\figurename}{Fig.}

\title[$l^{2}$-decoupling and the Boltzmann equation]{$l^{2}$-decoupling and the unconditional uniqueness for the Boltzmann equation}

\author[X. Chen]{Xuwen Chen}
\address{Department of Mathematics, University of Rochester, Rochester, NY 14627, USA}
\email{xuwenmath@gmail.com}

\author[S. Shen]{Shunlin Shen}
\address{School of Mathematical Sciences, University of Science and Technology of China, Hefei, 230026, China}
\email{slshen@ustc.edu.cn}

\author[Z. Zhang]{Zhifei Zhang}
\address{School of Mathematical Sciences, Peking University, Beijing, 100871, China}

\email{zfzhang@math.pku.edu.cn}

\subjclass[2010]{Primary 76P05, 35Q20, 35A02; Secondary 11L03,52C35,82C40.}

\dedicatory{}

\begin{abstract}
We broaden the application of the $l^{2}$-decoupling theorem to the Boltzmann equation. We prove  Strichartz estimates for the linear problem in the $\mathbb{T}^d$ setting. We establish space-time bilinear estimates, and hence the unconditional uniqueness of solutions to the $\mathbb{R}^d$ and $\mathbb{T}^d$ Boltzmann equation for the Maxwellian particle and soft potential with an angular cutoff, adopting a unified hierarchy scheme originally developed for the nonlinear Schr\"{o}dinger equation.

 \end{abstract}
\keywords{Unconditional Uniqueness, Boltzmann Hierarchy, Klainerman-Machedon Board
Game, $l^{2}$-decoupling}
\maketitle
\tableofcontents

\section{Introduction}
The Boltzmann equation in statistical mechanics describes the time-evolution of the distribution function of particles in a thermodynamic system away from equilibrium in the mesoscopic regime. It is a nonlinear integro-differential equation
\begin{equation}\label{equ:Boltzmann}
\left\{
\begin{aligned}
\left( \partial_t + v \cdot \nabla_x \right) f (t,x,v) =& Q(f,f),\\
f(0,x,v)=& f_{0}(x,v),
\end{aligned}
\right.
\end{equation}
where $f(t,x,v)$ denotes the distribution function of particles at time
 $t\geq 0$, located at position $x\in \R^{d}$, and with velocity $v\in \R^{d}$.
 The collision operator $Q(f,g)$, which describes the effect of binary collisions between particles, is conventionally
decomposed into a gain term
and a loss term
\begin{equation*}
Q(f,g)=Q^{+}(f,g)-Q^{-}(f,g)
\end{equation*}
where the gain term is
\begin{align}
Q^{+}(f,g)=&\int_{\mathbb{R}^{d}}\int_{\mathbb{S}^{d-1}} f(v^{\ast })g(u^{\ast}) B(u-v,\omega)dud\omega,
\end{align}
and the loss term is
\begin{align}
Q^{-}(f,g)=& \int_{\mathbb{R}^{d}}\int_{\mathbb{S}^{d-1}} f(v)g(u) B(u-v,\omega) du d\omega,
\end{align}
with the pre-collision and post-collision velocities related by
\begin{align*}
u^{*}=u+\omega\cdot (v-u) \omega,\quad v^{*}=v-\omega\cdot(v-u)\omega.
\end{align*}

The Boltzmann collision kernel function $B(u-v,\omega)$ is a non-negative function which depends solely on the relative velocity $|u-v|$ and the deviation angle $\theta$, with $\cos \theta:=\frac{u-v}{|u-v|}\cdot \omega$.
We consider the kernel function in the form
\begin{align}\label{equ:kernel function}
B(u-v,\omega)=|u-v|^{\ga}\textbf{b}(\cos \theta)
\end{align}
 under the Grad's angular cutoff assumption
\begin{align*}
0\leq  \textbf{b}(\cos \theta)\leq C|\cos\theta|.
\end{align*}
The collision kernel \eqref{equ:kernel function} originates from the physical model of inverse-power law potentials and
 the different ranges $\ga<0$, $\ga=0$, $\ga>0$ correspond to soft potentials, Maxwellian molecules, and hard potentials, respectively. See also \cite{cercignani1988boltzmann,cercignani1994mathematical,villani2002review} for a more detailed physics background.

In the paper, we consider the Cauchy problem of \eqref{equ:Boltzmann}, particularly the unconditional uniqueness of solutions in the low regularity setting. The well-posedness problem is one of the fundamental mathematical problems in kinetic theory, as it is of vital importance for both the physical interpretation and practical application of the theory.
 In many interesting physical situations, such as the study of rarefied gases, the initial data may have highly irregular initial particle distributions and thus are of low regularity. Furthermore, low regularity well-posedness theories, away from mathematical interest, are important for numerical simulations. In practical applications, numerical methods often need to handle non-smooth and low regularity data in each iteration, and requires a persistence of regularity of the solution map to ensure the validity of the algorithm. Understanding the behavior of these solutions at low regularity is not only vital for achieving accurate modeling and prediction, but also plays a crucial role in the development of more robust and precise numerical algorithms.

To date, a large quantity of mathematical theories and diverse methods have been developed for constructing solutions to \eqref{equ:Boltzmann} in various settings. See for example \cite{alexandre2013local,arsenio2011global,chen2019local,chen2019moments,chen2021small,CH24well,CSZ24well,diperna1989cauchy,
duan2016global,duan2018solution,guo2003classical,guo2003the,he2023cauchy,illner84the,kaniel78the,
alexandre2011the,alexandre2011thehard,alexandre2012the,alexandre2011global,chaturvedi2021stability,
duan2021global,gressman2011global,gressman2011sharp,imbert2022global}.
There have been many advancements of well-posedness theories on least regularity as possible on the initial data.
At the same time, it is highly nontrivial to find the critical regularity of initial data for well-posedness. On the one hand,
 it is sometimes believed at $s=\frac{d}{2}$, the continuity threshold, based on the failure of the critical embedding $H^{\frac{d}{2}}\hookrightarrow L^{\infty}$. See for example \cite{alexandre2013local,duan2016global,duan2021global,duan2018solution} for a more discussion. On the other hand, from the perspective of scaling analysis, the well/ill-posedness threshold in $H^{s}$ Sobolev space seems to be $s=\frac{d-2}{2}$, because \eqref{equ:Boltzmann} in $\R^{d}\times \R^{d}$ is scale-invariant at this regularity level.
This aspect studies solutions that may not be in $L^{\infty}$, and thus new ideas are needed.

In a series of work \cite{chen2019local,chen2019moments,chen2021small}, T. Chen, Denlinger, and Pavlovi$\acute{\text{c}}$ have provided a novel approach to establish the well-posedness of \eqref{equ:Boltzmann}, by adopting dispersive techniques from the study of the quantum many-body hierarchy dynamics, especially space-time collapsing/multi-linear estimates techniques (see for instance \cite{chen2010energy,chen2015unconditional,
chen2014derivation,chen2013rigorous,chen2016focusing,chen2016collapsing,chen2016klainerman,chen2016correlation,
chen2019derivation,chen2022quantitative,chen2023derivation,herr2016gross,herr2019unconditional,
kirkpatrick2011derivation,KM08,sohinger2015rigorous}).
 They have reduced the regularity for the well-posedness of \eqref{equ:Boltzmann} to $s>\frac{d-1}{2}$ for both Maxwellian molecules and hard potentials with cutoff, and also indicated the potential for a systematic study of \eqref{equ:Boltzmann} utilizing dispersive tools.
For the 3D constant kernel case, X. Chen and Holmer \cite{CH24well} identified that the threshold for well-posedness and ill-posedness in the $H^s$ Sobolev space is in fact at the regularity level $s=\frac{d-1}{2}$, which is between the continuous threshold $s=\frac{d}{2}$ and the scaling-invariant index $s=\frac{d-2}{2}$.
Building upon this, in our subsequent work \cite{CSZ24well}, we moved forward from the special constant kernel case, and proved that the well/ill-posedness threshold was also $s=\frac{d-1}{2}$ for the general
kernel with soft potentials. Furthermore, we established in \cite{chen2023sharp} a sharp 3D global well-posedness for both Maxwellian molecules and soft potential case. (Here, "sharp" means that the smallness is taken in the scaling-invariant space with $s=\frac{d-2}{2}$ and the regularity is just above the well/ill-posedness threshold.)

With the discovery of a sense of optimal regularity for well-posedness, numerous intriguing and important problems have arisen. One such problem is the uniqueness of solutions to \eqref{equ:Boltzmann}, particularly in such low regularity settings that are below the continuity threshold. Beyond its physical relevance and implications for numerical simulations,
the uniqueness problem in this contexts is not merely of mathematical interest, but also represents an actual necessity for addressing many related problems. For instance, it is crucial for the derivation of the Boltzmann equation from classical particle systems or quantum many-body dynamics \cite{chen2023derivationboltzmann}, as well as for understanding its hydrodynamic limit to fluid equations (for which the uniqueness has attracted much attention recently) and numerous other applications. Undoubtedly, characterizing the optimal regularity conditions that determine uniqueness or non-uniqueness presents a significant challenge.

 In this paper, we employ dispersive techniques and a hierarchy scheme from quantum many-body dynamics, to tackle the problem of unconditional uniqueness for \eqref{equ:Boltzmann} with the spatial domain being $\R^{d}$ and $\T^{d}$. In many applications of \eqref{equ:Boltzmann}, a periodic spatial domain is physically meaningful, as it sort of corresponds to the thermodynamic limit of particle systems and also provides a structured framework for modeling complex systems. However, from the perspective of dispersive equations, the analysis of low regularity problems on the periodic domain is much more delicate, akin to the case of nonlinear Schr\"{o}dinger equations on $\T^{d}$. In fact, to address this problem on $\T^{d}$, the $l^{2}$-decoupling theorem plays a crucial role.
Here is the main theorem.

\begin{theorem}\label{thm:main theorem}
Let $d\geq 2$, $p_{0}=\frac{2(d+2)}{d}$, $r>\frac{d}{2}+\ga$, and
\begin{equation}\label{equ:index,thm}
\left\{
\begin{aligned}
&s>\frac{d-1}{2}, \quad \ga\in [\frac{1-d}{2},0],\quad &\text{on $\R^{d}\times \R^{d}$},\\
&s>\frac{d}{2}-\frac{1}{p_{0}},\quad \ga\in [-\frac{d}{p_{0}},0] ,\quad& \text{on $\T^{d}\times \R^{d}$.}
\end{aligned}
\right.
\end{equation}
There is at most one $C([0,T];H_{x}^{s}L_{v}^{2,r})$ solution to the Boltzmann equation \eqref{equ:Boltzmann}. Here, the weighted Sobolev norm $H_{x}^{s}L_{v}^{2,r}$ is given by
\begin{align}
\n{f}_{H_{x}^{s}L_{v}^{2,r}}:=\n{\lra{\nabla_{x}}^{s}\lra{v}^{r}f}_{L_{x,v}^{2}}.
\end{align}
\end{theorem}

For the $\R^{d}\times \R^{d}$ case, Theorem \ref{thm:main theorem} provides a sharp uniqueness result for the Boltzmann equation in the sense that it matches the optimal regularity for the local well-posedness theories as in \cite{CH24well,CSZ24well} in which the ill-posedness result is proved.
For the $\T^{d}\times \R^{d}$ case, the result should represent the best achievable by our method. This is primarily due to the fact that the Strichartz estimates on $\T^{d}\times \R^{d}$ are significantly weaker than those on $\R^{d}\times \R^{d}$. Typically, the $L_{t}^{2}$ scale-invariant endpoint Strichartz estimate does not hold in the periodic domain. Instead, a weaker scale-invariant $L_{t}^{p}$ with a frequency cutoff only holds for $p\geq p_{0}$ where $p_{0}=\frac{2(d+2)}{2}$ is the critical point for the symmetric hyperbolic Schr\"{o}dinger in \eqref{equ:hyperbolic equation linear}.
This difference in the strength of Strichartz estimates accounts for the disparity in results \eqref{equ:index,thm} between
 $\R^{d}$ and $\T^{d}$.

To elaborate,
we begin by establishing the connection between the analysis of the Boltzmann equation \eqref{equ:Boltzmann} and the theory of dispersive PDEs.
Let $\wt{f}(t,x,\xi)$ be the inverse Fourier transform in the velocity variable, defined as
\begin{align*}
\wt{f}(t,x,\xi)=\mathcal{F}_{v\mapsto \xi}^{-1}(f).
\end{align*}
Then the linear part of \eqref{equ:Boltzmann} is transformed into the symmetric hyperbolic Schr\"{o}dinger equation\footnote{This name might also mean $-\Delta_{x_{1}}+\Delta_{x_{2}}$ with $x_{1}$ and $x_{2}$ in the same spatial domain, in contrast with the elliptic Laplace operator $-\Delta_{x_{1}}-\Delta_{x_{2}}$.}
\begin{align}\label{equ:hyperbolic equation linear}
i\pa_{t}\wt{f}+\nabla_{\xi}\cdot \nabla_{x}\wt{f}=0,
\end{align}
which, based on the endpoint Strichartz estimates in \cite[Theorem 1.2]{keel1998endpoint} for the $\R^{d}\times \R^{d}$ case, enables the application of Strichartz estimates that
\begin{align}\label{equ:strichartz,R,thm}
\n{e^{it\nabla_{\xi}\cdot \nabla_{x}}\wt{f}_{0}}_{L_{t}^{q}L_{x,\xi}^{p}(\R\times \R^{d}\times \R^{d})}\lesssim \n{\wt{f}_{0}}_{L_{x,\xi}^{2}(\R^{d}\times \R^{d})},\quad \frac{2}{q}+\frac{2d}{p}=d,\quad q\geq 2,\ d\geq2.
\end{align}
For the $\T^{d}\times \R^{d}$ case, \eqref{equ:strichartz,R,thm} is no longer applicable. Instead, to recover  Strichartz estimates, we heavily rely on the $l^{2}$-decoupling theorem in Bourgain and Demeter \cite{BD15}. The $l^{2}$-decoupling theorem is a groundbreaking result in harmonic analysis, with profound implications across various fields of mathematics. It has led to significant advancements in the understanding of PDEs, number theory, geometric measure theory. In the context, we broaden the application of the $l^{2}$-decoupling theorem to the Boltzmann equation by establishing its connection with Stricharts estimates.
\begin{proposition}[Strichartz estimates with separated $(x,\xi)$ frequencies]
Let $p\geq p_{0}=\frac{2(d+2)}{d}$. We have
\begin{align}\label{equ:strichartz,different frequency,thm}
&\n{P_{\leq N}^{x}P_{\leq M}^{\xi}e^{i t \nabla_{\xi}\cdot \nabla_{x}} \wt{f}_{0}}_{L_{t}^p([0,1]\times \mathbb{T}^d\times \R^{d})}\\
\lesssim_{\ve}&\max\lr{1,MN^{-1}}^{\frac{1}{p}}N^{\frac{d}{2}-\frac{d+1}{p}+\ve}
M^{\frac{d}{2}-\frac{d+1}{p}}\n{P_{\leq N}^{x}P_{\leq M}^{\xi}\wt{f}_{0}}_{L_{x,\xi}^{2}}.\notag
\end{align}
\end{proposition}
The  Strichartz estimates \eqref{equ:strichartz,different frequency,thm} are sharp up to $\ve$-losses. They are not only essential for proving the space-time bilinear estimates required for our proof, but also hold independent interest. They may potentially be applied to a broader range of problems, such as the kinetic cascade problem and the stability analysis around the Maxwell distribution, among others. Also see \cite{bacsakouglu2024local}, in which the $L_{t,x,\xi}^{4}$ Strichartz estimates have been established to achieve the local-wellposedness of the Boltzmann equation with the periodic spatial domain. Notably, although the Strichartz estimates \eqref{equ:strichartz,different frequency,thm} appear similar in form to those for the Schr\"{o}dinger operator
 $e^{it\Delta_{x}+it\Delta_{\xi}}$ on the waveguide $\T^{d}\times \R^{d}$, they are fundamentally different.
 \footnote{The critical point is a key difference between the two operators.  For example, for $d=3$, the critical point for the symmetric hyperbolic Schr\"{o}dinger operator is $p_{0}=\frac{10}{3}$, while for the Schr\"{o}dinger operator $e^{it\Delta_{x}+it\Delta_{\xi}}$, it is $p_{1}=\frac{8}{3}$.} (See Remark \ref{remark:symmetry and scaling} for further discussions.)

In addition to the dispersive property, another characteristic of the Boltzmann equation is its hierarchy structure, which
is a linear system of infinitely many
coupled differential equations for the correlation functions of a rarefied gas
of particles.
 In \cite{ACI91},
Arkeryd, Caprino and Ianiro have suggested using the Hewitt-Savage theorem to prove uniqueness for the Boltzmann hierarchy. Here, we adopt a perfected hierarchy scheme for the nonlinear Schr\"{o}dinger equation (NLS) case in the quantum setting.

The hierarchy approach for NLS, proposed by Spohn \cite{spohn1980kinetic}, aims to derive the NLS from quantum many-body dynamics.
Around 2005, it was Erd{\"o}s, Schlein, and Yau who first rigorously derived the 3D cubic defocusing NLS in their fundamental papers \cite{erdos2006derivation,erdos2007derivation,erdos2009rigorous,erdos2010derivation}
with the unconditional uniqueness of the Gross-Pitaevskii (GP) hierarchy being a key part. Subsequently,
Klainerman and Machedon \cite{KM08}
introduced a different uniqueness theorem, and provided a different combinatorial argument, the now so-called Klainerman-Machedon (KM) board game.
 In 2013, T. Chen, Hainzl, Pavlovi{\'c}, and Seiringer \cite{chen2015unconditional} simplified the proof of the uniqueness theorem in \cite{erdos2007derivation} by applying the quantum de Finetti theorem to the KM board game. Their method, combining the KM board game, quantum de Finetti theorem, and Sobolev multilinear estimates, is robust for such uniqueness problems. The uniqueness analysis of GP hierarchy started to unexpectedly yield new NLS results with regularity lower than the NLS analysis all of a sudden since Herr-Sohinger
  \cite{herr2019unconditional}\footnote{We mention \cite{herr2019unconditional} 1st here. Even though
 \cite{chen2019derivation} was posted on arXiv one month before
 \cite{herr2019unconditional}, X. Chen and Holmer were not
 aware of the unconditional uniqueness implication of \cite{chen2019derivation}
 until \cite{herr2019unconditional}.} and \cite{chen2019derivation,chen2022unconditional,CSZ22} in which the unconditional uniqueness problems for $H^{1}$-critical and $H^{1}$-supercritical cubic and quintic NLS are completely resolved on $\R^{d}$ and $\T^{d}$. (See also \cite{Kis21} using NLS analysis in the scaling-subcritical regime.) Currently, the hierarchy scheme appears to be the only known approach to deal with the uniqueness problem on $\T^{d}$ at the critical regularity.

Through a great deal of efforts, the hierarchy scheme within the quantum framework has emerged as a robust approach. Applying the developed techniques to kinetic equations is a meaningful endeavor.
In the study of kinetic equations, T. Chen, Denlinger, and Pavlovi{\'c}  pioneered the KM board game method in \cite{chen2019local}, demonstrating the local well-posedness of the Boltzmann hierarchy for cutoff Maxwell molecules. More recently, X. Chen and Holmer restructured this KM board combinatorics in \cite{chen2023derivationboltzmann}, integrating the latest techniques from \cite{chen2022unconditional} (the supercritical case is in \cite{CSZ22}), to achieve sharp unconditional uniqueness for the Boltzmann hierarchy with a collision kernel that combines aspects of hard sphere and inverse power potential. (See also \cite{ampatzoglou2024global} for the $L^{\infty}$ setting.)

\vspace{1em}

\noindent \textbf{Outline of the paper.}
We prove Theorem \ref{thm:main theorem} by adopting the recently developed hierarchy scheme for the NLS case in \cite{chen2019derivation,chen2022unconditional,CSZ22}.
 Although the Boltzmann equation and the NLS share similarities in the context of dispersive equations and infinite hierarchies, they are still very different.
 The difficulty primarily arises from the complexity of the Boltzmann collision kernel,
 which, as a key component of the Boltzmann equation, involves a two-fold integral
 and exhibits highly nonlinear, nonlocal, and nonsymmetric properties.

 Within the hierarchy method framework, a pivotal observation is that the iteration part in Section \ref{sec:Iterative Estimates} essentially reduces to symmetric collision estimates \eqref{equ:Q,bilinear estimate,Hs,1}--\eqref{equ:Q,bilinear estimate,Hs,2}.
In general, establishing symmetric estimates for the Boltzmann equation is more delicate than for the NLS (in the low-regularity setting), due to the intrinsic nonsymmetry and the different type of nonlinearity.
The multilinear estimates for both equations are often proved by a case-by-case Littlewood-Paley frequency analysis. However, the nonsymmetry of the Boltzmann collision term introduces additional frequency interaction case that are absent in more symmetric settings, creating extra complications. (See, for example, the distinct treatment required for Case C in the analysis of the loss and gain terms.) To overcome this, we employ a more refined decomposition in the frequency space. We remark that, this nonsymmetry is not merely a technical obstacle. It has in fact played a crucial role in understanding the ill-posedness mechanism of the Boltzmann equation, as demonstrated in \cite{CH24well, CSZ24well}.

In the analysis of the Boltzmann hierarchy, we supply
a Duhamel tree representation, with an algorithm to connect the tree structure to the Duhamel expansion.
Together with the KM board game combinatorics, this diagrammatic representation yields a streamlined iterative scheme in the $L^{2}$ setting. This enables a direct application of the symmetric collision estimates \eqref{equ:Q,bilinear estimate,Hs,1}--\eqref{equ:Q,bilinear estimate,Hs,2} for the nonsymmetric collision operator, thereby providing a unified treatment for both $\R^{d}$ and $\T^{d}$ spatial domains.
In this way, the analysis successfully integrates tools from modern harmonic analysis, specifically decoupling inequalities, dispersive estimates for symmetric hyperbolic Schr\"{o}dinger flows, into the Boltzmann hierarchy scheme.

In Section \ref{sec:Strichartz Estimates}, we focus on the proof of the  Strichartz estimates on $\T^{d}\times \R^{d}$ via the $l^{2}$-decoupling theorem.
It should be noted that,
unlike the elliptic Schr\"{o}dinger operator $e^{it\Delta_{x}+ it\Delta_{\xi}}$ on the waveguide $\T^{d}\times \R^{d}$, in which partial derivatives can be separated, the symmetric hyperbolic Schr\"{o}dinger operator $e^{it\nabla_{\xi}\cdot \nabla_{x}}$ does not allow for such separation, making it more difficult to handle. Indeed, those separations can even yield global-in-time Strichartz estimates for the elliptic Schr\"{o}dinger operator on the waveguide, while it is impossible for the symmetric hyperbolic Schr\"{o}dinger operator. (See Remark \ref{remark:symmetry and scaling}.)
Here, one key observation is that the symmetric hyperbolic Schr\"{o}dinger operator possesses a scaling property given by
$$e^{it\nabla_{\xi}\cdot \nabla_{x}}(\phi(t,x,a\xi))=(e^{iat\nabla_{\xi}\cdot \nabla_{x}}\phi)(t,x,a\xi),$$
which allows us to perform scaling analysis on both the $\xi$ variable and the time variable $t$.
This property is in fact helpful for our handling of this new, non-elliptic, non-waveguide, highly symmetric and hyperbolic case in an unusual way.

In Section \ref{sec:Bilinear Estimates}, we employ dispersive techniques in conjunction with the Strichartz estimates to derive the space-time bilinear estimates for the collision kernel. Though the loss term and the gain term scale the same way, they have totally different structures. Therefore, we divide the proof into Sections \ref{sec:Bilinear Estimates for the Loss Term} and \ref{sec:Bilinear Estimates for the Gain Term}.

 In Section \ref{sec:Uniqueness of the Infinite Boltzmann Hierarchy}, we introduce the Boltzmann hierarchy and adopt the hierarchy scheme along with the KM board game combinatorics method to obtain the uniqueness of the Boltzmann hierarchy and equation. Specifically, in Section \ref{sec:Duhamel Expansion and Duhamel Tree}, we start with an algorithm to generate a Duhamel tree diagram representing the Duhamel expansion. Then in Section \ref{sec:Iterative Estimates}, based on the Duhamel tree diagram, we prove the iterative estimates for the Duhamel expansion. Finally, in Section \ref{sec:Proof of the Main Theorem}, we complete the proof of the main theorem.

\section{Strichartz Estimates via $l^{2}$-decoupling}\label{sec:Strichartz Estimates}

In the section, our main goal is to use the $l^{2}$-decoupling theorem in \cite{BD15} to set up the  Strichartz estimates with separated $(x,\xi)$ frequencies for the symmetric hyperbolic Schr\"{o}dinger operator.
Following \cite{BD15,BD17}, we introduce the truncated hyperbolic paraboloid
\begin{align}
H_{\al}^{n-1}:=\lr{(\eta_{1},...,\eta_{n-1},\al_{1}\eta_{1}^{2}+...+\al_{n-1}\eta_{n-1}^{2}),
|\eta_{i}|\leq \frac{1}{2}},
\end{align}
where $\al=(\al_{1},...,\al_{n})$ with $\al_{i}=\pm 1$.

Let $\mathcal{N}_\delta=\mathcal{N}_\delta(H_{\al}^{n-1})$ be the $\delta$ neighborhood of $H_{\al}^{n-1}$, and let $\mathcal{P}_\delta$ be a finitely overlapping cover of $\mathcal{N}_\delta$ with $\sim \delta^{1 / 2} \times \ldots \delta^{1 / 2} \times \delta$ rectangular boxes $\theta$ of the form
$$
\theta=\left\{\left(\eta_1, \ldots, \eta_{n-1}, \eta+\al_{1}\eta_1^2+\ldots+\al_{n-1}\eta_{n-1}^2\right):\left(\eta_1, \ldots, \eta_{n-1}\right) \in C_\theta,|\eta| \leq 2 \delta\right\}
$$
where $C_\theta$ runs over all cubes $c+\left[-\frac{\delta^{1 / 2}}{2}, \frac{\delta^{1 / 2}}{2}\right]^{n-1}$ with $c \in \frac{\delta^{1 / 2}}{2} \mathbb{Z}^{n-1} \cap[-1 / 2,1 / 2]^{n-1}$. Note that each $\theta$ sits inside a $\sim \delta^{1 / 2} \times \ldots \delta^{1 / 2} \times \delta$ rectangular box. We denote by $P_{\theta}\phi$ the Fourier restriction of $\phi$ to $\theta$.

\begin{lemma}[$l^{2}$-decoupling, \cite{BD15,BD17}]\label{lemma:l2 decoupling}
If $\operatorname{supp}(\widehat{\phi})\subset \mathcal{N}_{\delta}(H_{\al}^{n-1})$, then for
$\ve>0$,
\begin{align}
&\n{\phi}_{L^{p}(\R^{n})}\lesssim_{\ve}\delta^{-\frac{n-1}{4}+\frac{n+1}{2p}-\ve}
\lrs{\sum_{\theta\in \mathcal{P}_{\delta}}\n{P_{\theta}\phi}_{L^{p}(\R^{n})}^{2}}^{\frac{1}{2}}, \quad \text{for $p\geq \frac{2(n+1-d(\al))}{n-1-d(\al)}$},\label{equ:l2 decoupling, p>p0}\\
&\n{\phi}_{L^{p}(\R^{n})}\lesssim_{\ve}\delta^{-\frac{d(\al)}{4}+\frac{d(\al)}{2p}-\ve}
\lrs{\sum_{\theta\in \mathcal{P}_{\delta}}\n{P_{\theta}\phi}_{L^{p}(\R^{n})}^{2}}^{\frac{1}{2}}, \quad \text{for $p\in [2,\frac{2(n+1-d(\al))}{n-1-d(\al)}$}],
\end{align}
where $d(\al)$ is the minimum of the number of
positive and negative entries of $\al$.
\end{lemma}
\begin{remark}
As pointed out by \cite{BD17}, in the hyperbolic case $d(\al)\geq 1$, the exponent is sharp but always nonzero, in contrast with the continuous case and the elliptic case $d(\al)=0$.
\end{remark}

Next, we get into the analysis of Strichartz estimates.
Let $\chi(x)$ be a smooth function that satisfies $\chi(x) = 1$ for all $|x| \leq 1$ and $\chi(x) = 0$ for $|x| \geq 2$. Let $N$ be a dyadic number, and define $\varphi_{N}(x) = \chi\left(\frac{x}{N}\right) - \chi\left(\frac{x}{2N}\right)$. The Littlewood-Paley projector is then given by
\begin{align}
\widehat{P_{N}u}(\eta) = \varphi_{N}(\eta) \widehat{u}(\eta).
\end{align}
We denote by $P_{N}^{x}$/$P_{M}^{\xi}$ the projector of the $x$-variable and $\xi$-variable respectively.
\begin{proposition}[ Strichartz estimates with separated $(x,\xi)$ frequencies]\label{lemma:strichartz estimate,different frequency}
Let $p_{0}=\frac{2(d+2)}{d}$.
For $p\geq p_{0}$,
\begin{align}\label{equ:strichartz,different frequency}
&\n{P_{\leq N}^{x}P_{\leq M}^{\xi}e^{i t \nabla_{\xi}\cdot \nabla_{x}} \wt{f}_{0}}_{L_{t}^p([0,1]\times \mathbb{T}^d\times \R^{d})}\\
\lesssim_{\ve}&\max\lr{1,MN^{-1}}^{\frac{1}{p}}N^{\frac{d}{2}-\frac{d+1}{p}+\ve}
M^{\frac{d}{2}-\frac{d+1}{p}}\n{P_{\leq N}^{x}P_{\leq M}^{\xi}\wt{f}_{0}}_{L_{x,\xi}^{2}}.\notag
\end{align}

\end{proposition}

\begin{proposition}\label{lemma:strichartz estimate,different frequency,2}
For $p\in[2,p_{0}]$,
\begin{align}\label{equ:strichartz,different frequency,p<p0}
&\n{P_{\leq N}^{x}P_{\leq M}^{\xi}e^{i t \nabla_{\xi}\cdot \nabla_{x}} \wt{f}_{0}}_{L_{t}^p([0,1]\times \mathbb{T}^d\times \R^{d})}\\
\lesssim_{\ve}&\max\lr{1,MN^{-1}}^{\frac{1}{p}}N^{\frac{1}{p}+\ve}
M^{\frac{d}{2}-\frac{d+1}{p}}\n{P_{\leq N}^{x}P_{\leq M}^{\xi}\wt{f}_{0}}_{L_{x,\xi}^{2}}.\notag
\end{align}
\end{proposition}

\begin{remark}
For
$p>\frac{2(d+2)}{d}$, the term
$N^{\ve}$ in \eqref{equ:strichartz,different frequency} could potentially be removed. In the case of $e^{it\Delta}$ on $\T^{d}$, see \cite{bourgain1993fourier1,KV16}. However, since this refinement does not lead to an improved result for the bilinear estimates we need, we skip it for simplicity.
\end{remark}

\begin{remark}\label{remark:symmetry and scaling}
Let us provide more explanations regarding on the symmetric hyperbolic Schr\"{o}dinger operator $e^{it\nabla_{\xi}\cdot \nabla_{x}}$ on $\T^{d}\times \R^{d}$, which is different from the elliptic Schr\"{o}dinger operator $e^{it\Delta_{x}+ it\Delta_{\xi}}$ on $\T^{d}\times \R^{d}$. We might as well set $d=3$ and hence $p_{0}=\frac{10}{3}$.
\begin{itemize}
\item (Hyperbolic v.s. elliptic.)
It is pointed out by \cite{BD17} that the exponent of the $l^{2}$-decoupling theorem depends on the type of hypersurface. Consequently,
the critical points differ between the hyperbolic and elliptic cases. For the hyperbolic case, the critical point is $p_{0}=\frac{10}{3}$, while for the elliptic case, it is $p_{1}=\frac{8}{3}$. More specifically, the elliptic Strichartz estimate in \cite[Proposition 3.4]{Bar21} is
\begin{align*}
\n{e^{it\Delta_{x}+it\Delta_{\xi}}P_{\leq N}^{x}P_{\leq N}^{\xi}u}_{L^{p}([0,1]\times \T^{3}\times \R^{3})}\lesssim_{\ve} N^{\frac{3p-8}{p}+\ve}\n{P_{\leq N}^{x}P_{\leq N}^{\xi}u}_{L_{x,\xi}^{2}},\quad p\geq p_{1}=\frac{8}{3}.
\end{align*}
Interestingly, the point $p_{0}=\frac{10}{3}$
 also coincides with the critical point for $e^{it\it\Delta_{x}}$ on $\T^{3}$.
\item (Local v.s. global.) The global-in-time Strichartz estimates might be true for $e^{it\Delta_{x}+it\Delta_{\xi}}$ on the waveguide
$\T^{d}\times \R^{d}$. However, they cannot hold for symmetric hyperbolic Schr\"{o}dinger operator. This is because
    $$e^{it\nabla_{\xi}\cdot\nabla_{x}}(\psi(\xi))=\psi(\xi),$$
    which is time-independent.
     Consequently, any global-in-time Strichartz estimates cannot apply.

\item (Scaling property.) Unlike the elliptic or hyperbolic Schr\"{o}dinger operators $e^{it\Delta_{x}\pm it\Delta_{\xi}}$ on the waveguide $\T^{d}\times \R^{d}$, whose partial derivatives can be separated, the symmetric hyperbolic Schr\"{o}dinger operator $e^{it\nabla_{\xi}\cdot \nabla_{x}}$ does not allow for such separation, making it more hard to handle. However, it possesses a scaling property given by
$$e^{it\nabla_{\xi}\cdot \nabla_{x}}(\phi(t,x,a\xi))=(e^{iat\nabla_{\xi}\cdot \nabla_{x}}\phi)(t,x,a\xi),$$
which allows us to perform scaling analysis on both the $\xi$ variable and the time variable $t$.
This property is crucial for our proof of Strichartz estimates \eqref{equ:strichartz,different frequency}--\eqref{equ:strichartz,different frequency,p<p0}.
\end{itemize}
\end{remark}
\begin{remark}
For the sake of completeness, we present the Strichartz estimate \eqref{equ:strichartz,different frequency,p<p0} for a range of
$p\in [2,p_{0}]$. In this paper, we specifically utilize this estimate \eqref{equ:strichartz,different frequency,p<p0} within the range $[p_{0}-\ve,p_{0}]$ for sufficiently small $\ve>0$.
\end{remark}
\begin{proof}[\textbf{Proof of Propositions $\ref{lemma:strichartz estimate,different frequency}$ and
$\ref{lemma:strichartz estimate,different frequency,2}$}]
It suffices to prove \eqref{equ:strichartz,different frequency}, as \eqref{equ:strichartz,different frequency,p<p0} follows from a similar way.
We divide the entire proof of \eqref{equ:strichartz,different frequency} into the three steps.

\textbf{Step 1. Strichartz estimate with comparable $(x,\xi)$ frequency:}
\begin{align}\label{equ:strichartz,same frequency}
\n{P_{\leq N}^{x}P_{\leq N}^{\xi}e^{i t \nabla_{\xi}\cdot \nabla_{x}}\wt{f}_{0}}_{L_{t}^p([0,T]\times \mathbb{T}^d\times \R^{d})} \lesssim_{\ve} T^{\frac{1}{p}} N^{d-\frac{2d+2}{p}+\ve}\n{\wt{f}_{0}}_{L_{x,\xi}^2}.
\end{align}
Without loss of generality, we consider the case where $T=1$, as the general case $T\gg 1$ can be derived from the summation over the time intervals.

Let $\eta(t):\R\mapsto \mathbb{C}$ and $\zeta(x):\R^{d}\mapsto \mathbb{C}$ be two smooth functions such that
\begin{align}
&\text{$\operatorname{supp}\widehat{\eta}\subset (-1,1)$, and $|\eta(t)|\geq 1$ for $t\in (-1,1)$,}\\
&\text{$\operatorname{supp}\widehat{\zeta}\subset B(0,1)$, and $|\zeta(x)|\geq 1$ for $x\in B(0,1)$.}\label{equ:cutoff function,supp}
\end{align}
For simplicity, we also set
\begin{align*}
F_{N}(t,x,\xi)=P_{\leq N}^{x}P_{\leq N}^{\xi}e^{i t \nabla_{\xi}\cdot \nabla_{x}}\wt{f}_{0},\quad G_{N}(t,x,\xi)=\eta(t)\zeta(\frac{x}{N})F_{N}(t,x,\xi).
\end{align*}
Using the periodicity in the
$x$-variable, we obtain
\begin{align}\label{equ:FN,GN}
\n{F_{N}(t,x,\xi)}_{L^{p}_{t,x,\xi}([-1,1]\times \mathbb{T}^{d}\times \R^{d})}
\lesssim &N^{-\frac{d}{p}}\n{F_{N}(t,x,\xi)}_{L^{p}_{t,x,\xi}([-1,1]\times [-N,N]^{d}\times \R^{d})}\\
\lesssim
&N^{-\frac{d}{p}}\n{G_{N}(t,x,\xi)}_{L^{p}_{t,x,\xi}(\R\times \R^{d}\times \R^{d})}\notag\\
=&N^{-\frac{d}{p}}\n{G_{1,N}(t,x,\xi)}_{L^{p}_{t,x,\xi}(\R\times \R^{d}\times \R^{d})},\notag
\end{align}
where
\begin{align}\label{equ:rotation,scaling}
G_{1,N}(t,x,\xi)=N^{-\frac{2d+2}{p}}G_{N}(\frac{t}{N^{2}},\frac{x-\xi}{2N},\frac{x+\xi}{2N}).
\end{align}
Noticing that $\operatorname{supp}(\widehat{G_{1,N}})\subset \mathcal{N}_{\delta}(H_{\al}^{n-1})$ with
$n=2d+1$ and $\delta\sim N^{-2}$,
we can apply the $l^{2}$-decoupling in Lemma \ref{lemma:l2 decoupling} to $G_{1,N}(t,x,\xi)$, and hence obtain
\begin{align}\label{equ:estimate,G1N}
\n{G_{1,N}(t,x,\xi)}_{L^{p}_{t,x,\xi}(\R\times \R^{d}\times \R^{d})}\lesssim_{\ve}& N^{d-\frac{2d+2}{p}+\ve}
\lrs{\sum_{\theta\in \mathcal{P}_{\delta}} \n{P_{\theta}G_{1,N}}_{L_{t,x,\xi}^{p}}^{2} }^{\frac{1}{2}}\\
\lesssim & N^{d-\frac{2d+2}{p}+\ve}
\lrs{\sum_{|m|\leq N,|n|\leq N} \n{P_{\theta_{m,n}}G_{N}}_{L_{t,x,\xi}^{p}}^{2} }^{\frac{1}{2}},\notag
\end{align}
where $P_{\theta_{n,m}}$ is the projection onto the frequency set
\begin{align}\label{equ:theta,m,n,supp}
\theta_{m,n}=\lr{(\tau,y,v):y\in m+[-\frac{1}{2},\frac{1}{2}]^{d},v\in n+[-\frac{1}{2},\frac{1}{2}]^{d},|\tau-y\cdot v|\leq 1},\quad m,n\in \frac{\Z^{d}}{2}.
\end{align}
Alternatively, one could also apply $l^{2}$-decoupling inequalities for a broad
class of quadratic forms established in \cite[Corollary 1.3]{GOZZ23}
directly to the hyperbolic quadratic form $x\cdot \xi$, and thereby avoid the rotation in \eqref{equ:rotation,scaling}.

Then, using Hausdorff-Young inequality, the support conditions in \eqref{equ:cutoff function,supp} and \eqref{equ:theta,m,n,supp}, and H\"{o}lder inequality, we have
\begin{align}
&\n{P_{\theta_{m,n}}G_{N}}_{L_{t,x,\xi}^{p}}\\
=&\n{P_{\theta_{m,n}}\lrc{\eta(t)\zeta(\frac{x}{N})P_{\leq N}^{x}P_{\leq N}^{\xi}e^{i t \nabla_{\xi}\cdot \nabla_{x}}\wt{f}_{0}}}_{L_{t,x,\xi}^{p}}\notag\\
\lesssim& \bbn{1_{\theta_{m,n}}\lrc{\sum_{|k|\leq N}\widehat{\eta}(\tau-k\cdot v)N^{d}\widehat{\zeta}(N(y-k))\chi_{N}(v)\widehat{\wt{f}_{0}}(k,v)}}_{L_{\tau,y,v}^{p'}}\notag\\
\lesssim& \sum_{|k-m|\lesssim 1}\bbn{1_{\theta_{m,n}}\lrc{\widehat{\eta}(\tau-k\cdot v)N^{d}\widehat{\zeta}(N(y-k))\chi_{N}(v)\widehat{\wt{f}_{0}}(k,v)}}_{L_{\tau,y,v}^{p'}}\notag\\
\lesssim& \sum_{|k-m|\lesssim 1}\n{\widehat{\eta}}_{L_{\tau}^{p'}}\n{N^{d}\widehat{\zeta}(N (y-k))}_{L_{y}^{p'}}\n{1_{\lr{n+[-\frac{1}{2},\frac{1}{2}]^{d}}}(v)\widehat{\wt{f}_{0}}(k,v)}_{L_{v}^{p'}}\notag\\
\lesssim& \sum_{|k-m|\lesssim 1} N^{\frac{d}{p}}\n{1_{\lr{n+[-\frac{1}{2},\frac{1}{2}]^{d}}}(v)\widehat{\wt{f}_{0}}(k,v)}_{L_{v}^{2}}.\notag
\end{align}
Therefore, by the $L^{2}$ almost orthogonality, we arrive at
\begin{align}\label{equ:m,n,Gn,estimate}
\lrs{\sum_{|m|\leq N,|n|\leq N} \n{P_{\theta_{m,n}}G_{N}}_{L_{t,x,\xi}^{p}}^{2} }^{\frac{1}{2}}\lesssim& N^{\frac{d}{p}}\lrc{
\sum_{|m|\leq N,|n|\leq N} \sum_{|k-m|\lesssim 1} \n{1_{\lr{n+[-\frac{1}{2},\frac{1}{2}]^{d}}}(v)\widehat{\wt{f}_{0}}(k,v)}_{L_{v}^{2}}^{2}}^{\frac{1}{2}}\\
\lesssim&N^{\frac{d}{p}}\n{\wt{f}_{0}}_{L_{x,\xi}^{2}}.\notag
\end{align}

Putting \eqref{equ:FN,GN}, \eqref{equ:estimate,G1N}, and \eqref{equ:m,n,Gn,estimate} together, we obtain
\begin{align}
\n{F_{N}(t,x,\xi)}_{L^{p}_{t,x,\xi}([0,1]\times \mathbb{T}^{d}\times \R^{d})}\lesssim_{\ve}
N^{d-\frac{2d+2}{p}+\ve} \n{\wt{f}_{0}}_{L_{x,\xi}^{2}},
\end{align}
and complete the proof of Strichartz estimate \eqref{equ:strichartz,same frequency} with the comparable $(x,\xi)$ frequency.

\textbf{Step 2. Strichartz estimate with high-low $(x,\xi)$ frequency:}
\begin{align}\label{equ:strichartz,high-low frequency}
\n{P_{\leq N}^{x}P_{\leq 1}^{\xi}e^{i t \nabla_{\xi}\cdot \nabla_{x}}\wt{f}_{0}}_{L_{t}^p([0,T]\times \mathbb{T}^d\times \R^{d})}
\lesssim_{\ve}\max\lr{1,TN^{-1}}^{\frac{1}{p}} N^{\frac{d}{2}-\frac{d+1}{p}+\ve}\n{\wt{f}_{0}}_{L_{x,\xi}^{2}}.
\end{align}
To do this, we introduce the scaling operator
\begin{align*}
 (\delta_{M}^{\xi}g)(\xi)=g(M\xi),
\end{align*}
and have
\begin{align}\label{equ:scaling analysis}
\delta_{N}^{\xi}P_{\leq 1}^{\xi}=P_{\leq N}^{\xi}\delta_{N}^{\xi},
\quad \delta_{N}^{\xi}e^{it\nabla_{\xi}\cdot \nabla_{x}}=e^{\frac{it\nabla_{\xi}\cdot \nabla_{x}}{N}}\delta_{N}^{\xi}.
\end{align}
Therefore, by scaling and Strichartz estimate \eqref{equ:strichartz,same frequency} with the comparable $(x,\xi)$ frequency, we obtain
\begin{align*}
&\n{P_{\leq N}^{x}P_{\leq 1}^{\xi}e^{i t \nabla_{\xi}\cdot \nabla_{x}}\wt{f}_{0}}_{L_{t}^p([0,T]\times \mathbb{T}^d\times \R^{d})} \\
=&\n{P_{\leq N}^{x}\delta_{\frac{1}{N}}^{\xi}\delta_{N}^{\xi}P_{\leq 1}^{\xi}e^{i t \nabla_{\xi}\cdot \nabla_{x}}\wt{f}_{0}}_{L_{t}^p([0,T]\times \mathbb{T}^d\times \R^{d})} \\
=&\n{P_{\leq N}^{x}\delta_{\frac{1}{N}}^{\xi}P_{\leq  N}^{\xi}e^{\frac{i t \nabla_{\xi}\cdot \nabla_{x}}{N}}\delta_{N}^{\xi}\wt{f}_{0}}_{L_{t}^p([0,T]\times \mathbb{T}^d\times \R^{d})}\\
=&N^{\frac{d}{p}}N^{\frac{1}{p}}\n{P_{\leq N}^{x}P_{\leq  N}^{\xi}e^{i t \nabla_{\xi}\cdot \nabla_{x}}\delta_{N}^{\xi}\wt{f}_{0}}_{L_{t}^p([0,\frac{T}{N}]\times \mathbb{T}^d\times \R^{d})}\\
\lesssim_{\ve}&\max\lr{1,TN^{-1}}^{\frac{1}{p}}N^{\frac{d}{p}}N^{\frac{1}{p}} N^{d-\frac{2d+2}{p}+\ve}\n{\delta_{N}^{\xi}\wt{f}_{0}}_{L_{x,\xi}^{2}}\\
=&\max\lr{1,TN^{-1}}^{\frac{1}{p}} N^{\frac{d}{2}-\frac{d+1}{p}+\ve}\n{\wt{f}_{0}}_{L_{x,\xi}^{2}}.
\end{align*}

\textbf{Step 3. Strichartz estimate \eqref{equ:strichartz,different frequency} with separated $(x,\xi)$ frequency.}

To prove \eqref{equ:strichartz,different frequency}, using the scaling analysis \eqref{equ:scaling analysis} again and Strichartz estimate \eqref{equ:strichartz,high-low frequency} with high-low $(x,\xi)$ frequency, we get
\begin{align*}
&\n{P_{\leq N}^{x}P_{\leq M}^{\xi}e^{i t \nabla_{\xi}\cdot \nabla_{x}} \wt{f}_{0}}_{L_{t}^p([0,1]\times \mathbb{T}^d\times \R^{d})}\\
=&\n{P_{\leq N}^{x}P_{\leq M}^{\xi}\delta_{M}^{\xi}\delta_{\frac{1}{M}}^{\xi}
e^{i t \nabla_{\xi}\cdot \nabla_{x}} \wt{f}_{0}}_{L_{t}^p([0,1]\times \mathbb{T}^d\times \R^{d})}\\
=&\n{P_{\leq N}^{x}\delta_{M}^{\xi}P_{\leq 1}^{\xi}
e^{i t M\nabla_{\xi}\cdot \nabla_{x}} \delta_{\frac{1}{M}}^{\xi}\wt{f}_{0}}_{L_{t}^p([0,1]\times \mathbb{T}^d\times \R^{d})}\\
= &M^{-\frac{1}{p}-\frac{d}{p}}\n{P_{\leq N}^{x}P_{\leq 1}^{\xi}
e^{i t \nabla_{\xi}\cdot \nabla_{x}} \delta_{\frac{1}{M}}^{\xi}\wt{f}_{0}}_{L_{t}^p([0,M]\times \mathbb{T}^d\times \R^{d})}\\
\lesssim_{\ve}&\max\lr{1,MN^{-1}}^{\frac{1}{p}} N^{\frac{d}{2}-\frac{d+1}{p}+\ve}M^{-\frac{1}{p}-\frac{d}{p}} \n{\delta_{\frac{1}{M}}^{\xi}\wt{f}_{0}}_{L_{x,\xi}^{2}}\\
=&\max\lr{1,MN^{-1}}^{\frac{1}{p}}N^{\frac{d}{2}-\frac{d+1}{p}+\ve}
M^{\frac{d}{2}-\frac{d+1}{p}}\n{\wt{f}_{0}}_{L_{x,\xi}^{2}}.
\end{align*}
Hence, we complete the proof of \eqref{equ:strichartz,different frequency}.
\end{proof}

Using Proposition \ref{lemma:strichartz estimate,different frequency} and Littlewood-Paley theory, we arrive at the following corollary.
\begin{corollary}\label{lemma:regularity strichartz}
Let $p_{0}=\frac{2(d+2)}{d}$.
\begin{enumerate}[$(i)$]
\item Let $p\geq p_{0}$, $s_{p}>\frac{d}{2}-\frac{d+1}{p}$. There holds that
\begin{align}\label{equ:strichartz estimate,regularity}
&\n{P_{\leq N}^{x}e^{i t \nabla_{\xi}\cdot \nabla_{x}}\wt{f}_{0}}_{L_{t}^p([0,1]\times \mathbb{T}^d\times \R^{d})}
\lesssim N^{s_{p}}\n{P_{\leq N}^{x}\wt{f}_{0}}_{L_{x}^{2}H_{\xi}^{s_{p}+\frac{1}{p}}}.
\end{align}
\item Let $p\in [2,p_{0}]$, $s_{p}>\frac{d}{2}-\frac{d+1}{p}$.
There holds that
\begin{align}\label{equ:strichartz estimate,regularity,p>2}
&\n{P_{\leq N}^{x}e^{i t \nabla_{\xi}\cdot \nabla_{x}}\wt{f}_{0}}_{L_{t}^p([0,1]\times \mathbb{T}^d\times \R^{d})}
\lesssim N^{\frac{1}{p}+\ve}\n{P_{\leq N}^{x}\wt{f}_{0}}_{L_{x}^{2}H_{\xi}^{s_{p}+\frac{1}{p}}}.
\end{align}
\end{enumerate}
\end{corollary}

\section{Space-time Bilinear Estimates}\label{sec:Bilinear Estimates}

In this section, we establish the space-time bilinear estimates for the loss term and the gain term.
Taking the inverse $v$-variable Fourier transform on both side of \eqref{equ:Boltzmann}, we get
\begin{align}
i\pa_{t}\wt{f}+\nabla_{\xi}\cdot \nabla_{x}\wt{f}=i\mathcal{F}_{v\mapsto \xi}^{-1}\lrc{Q(f,f)}.
\end{align}
By the well-known Bobylev identity in a more general case, see for example \cite{alexandre2000entropy,CSZ24well,chen2023sharp,desvillettes2003use}, it holds that (up to an unimportant constant)
\begin{align}
\mathcal{F}_{v\mapsto \xi}^{-1}\lrc{Q^{-}(f,g)}(\xi)=& \n{\textbf{b}}_{L^{1}(\mathbb{S}^{d-1})}\int \frac{\wt{f}(\xi-\eta) \wt{g}(\eta)}{|\eta|^{d+\ga}}d\eta,\\
\mathcal{F}_{v\mapsto \xi}^{-1}\lrc{Q^{+}(f,g)}(\xi)=&\int_{\R^{d}\times \mathbb{S}^{d-1}}\frac{\wt{f}(\xi^{+}+\eta)\wt{g}(\xi^{-}-\eta)}{|\eta|^{d+\ga}}
\textbf{b}(\frac{\xi}{|\xi|}\cdot \omega)d\eta d\omega,
\end{align}
where $\xi^{+}=\frac{1}{2}\lrs{\xi+|\xi|\omega}$ and $\xi^{-}=\frac{1}{2}\lrs{\xi-|\xi|\omega}$.
For convenience, we take the notation that
\begin{align}\label{equ:Q,fourier}
\wt{Q}^{\pm}(\wt{f},\wt{g})=\mathcal{F}_{v\mapsto \xi}^{-1}\lrc{Q^{\pm}(f,g)}.
\end{align}
The following is the main proposition of the section.
\begin{proposition}[$L_{x}^{2}$ and $H_{x}^{s}$ bilinear estimates]
\label{lemma:bilinear estimates,loss,gain}
Let $d\geq 2$, $p_{0}=\frac{2(d+2)}{d}$, $r>\frac{d}{2}+\ga$, and
\begin{equation}
\left\{
\begin{aligned}
&s>\frac{d-1}{2}, \quad \ga\in [\frac{1-d}{2},0],\quad &\text{on $\R^{d}\times \R^{d}$},\\
&s>\frac{d}{2}-\frac{1}{p_{0}},\quad \ga\in [-\frac{d}{p_{0}},0] ,\quad& \text{on $\T^{d}\times \R^{d}$.}
\end{aligned}
\right.
\end{equation}
For $s_{1}\in[0,s]$, we have
\begin{align}
\n{\wt{Q}(U(t)\tilde{f}_{0}, U(t)\tilde{g}_{0})}_{L_{t}^{1} H_{x}^{s_{1}} H_{\xi}^{r}} \lesssim T^{\frac{1}{2}}\n{\wt{f}_{0}}_{H_{x}^{s_{1}}H_{\xi}^{r}}\n{\wt{g}_{0}}_{H_{x}^{s}H_{\xi}^{r}},\label{equ:Q,bilinear estimate,Hs,1}\\
\n{\wt{Q}(U(t)\tilde{f}_{0}, U(t)\tilde{g}_{0})}_{L_{t}^{1} H_{x}^{s_{1}} H_{\xi}^{r}} \lesssim T^{\frac{1}{2}}\n{\wt{f}_{0}}_{H_{x}^{s}H_{\xi}^{r}}\n{\wt{g}_{0}}_{H_{x}^{s_{1}}H_{\xi}^{r}},\label{equ:Q,bilinear estimate,Hs,2}
\end{align}
with $U(t)=e^{it\nabla_{\xi}\cdot \nabla_{x}}$.
\end{proposition}

Estimates \eqref{equ:Q,bilinear estimate,Hs,1} and \eqref{equ:Q,bilinear estimate,Hs,2} exhibit symmetry, despite the nonsymmetric nature of the loss and gain terms. This symmetry is crucial for the iterative estimates established in Section \ref{sec:Iterative Estimates}. Notably, when $s_{1}=s$, they are closed and could be used for local well-posedness theories. When $s=0$, they are necessary for addressing the problem of unconditional uniqueness in the hierarchy method.

We focus primarily on the proof of the more difficult $\T^{d}\times \R^{d}$ case. The $\R^{d}\times \R^{d}$ case, which allows the application of the endpoint $L_{t}^{2}L_{x,\xi}^{\frac{2d}{d-1}}$ Strichartz estimate, can be handled similarly.
In Section \ref{sec:Bilinear Estimates for the Loss Term}, we establish the space-time bilinear estimates for the loss term, while the corresponding estimates for the gain term are proven in Section \ref{sec:Bilinear Estimates for the Gain Term}.

\subsection{Bilinear Estimates for the Loss Term}\label{sec:Bilinear Estimates for the Loss Term}
The goal of the section is to prove the bilinear estimates for the loss term
\begin{align}
\n{\wt{Q}^{-}(U(t)\tilde{f}_{0}, U(t)\tilde{g}_{0})}_{L_{t}^{1} H_{x}^{s_{1}} H_{\xi}^{r}} \lesssim T^{\frac{1}{2}}\n{\wt{f}_{0}}_{H_{x}^{s_{1}}H_{\xi}^{r}}
\n{\wt{g}_{0}}_{H_{x}^{s}H_{\xi}^{r}},\label{equ:Q-,bilinear estimate,Hs,1}\\
\n{\wt{Q}^{-}(U(t)\tilde{f}_{0}, U(t)\tilde{g}_{0})}_{L_{t}^{1} H_{x}^{s_{1}} H_{\xi}^{r}} \lesssim T^{\frac{1}{2}}\n{\wt{f}_{0}}_{H_{x}^{s}H_{\xi}^{r}}
\n{\wt{g}_{0}}_{H_{x}^{s_{1}}H_{\xi}^{r}}.\label{equ:Q-,bilinear estimate,Hs,2}
\end{align}

We first provide a time-independent bilinear estimate for the loss term.
\begin{lemma}\label{lemma:Q-,bilinear estimate}
There holds that
\begin{align}\label{equ:Q-,bilinear estimate,f,g}
\n{\wt{Q}^{-}(\wt{f},\wt{g})}_{L_{\xi}^{2}}
\lesssim& \n{\wt{f}}_{L_{\xi}^{2}}
\n{\wt{g}}_{W^{r_{p},p}},
\end{align}
provided that $r_{p}>\frac{d}{p}+\ga>0$.
\end{lemma}

\begin{proof}
By Young's inequality and interpolation inequality, we have
\begin{align}\label{equ:Q-,bilinear estimate,f,g,step1}
\n{\wt{Q}^{-}(\wt{f},\wt{g})}_{L_{\xi}^{2}}\lesssim& \n{\wt{f}}_{L_{\xi}^{2}}
\bbn{\frac{\wt{g}(\xi)}{|\xi|^{d+\ga}}}_{L_{\xi}^{1}}
\lesssim\n{\wt{f}}_{L_{\xi}^{2}}
\n{\wt{g}}_{L^{p_{1}}}^{\theta}\n{\wt{g}}_{L^{p_{2}}}^{1-\theta},
\end{align}
where $1 \leq p_{1}<\frac{d}{-\gamma}<p_{2}\leq \infty$ and the weight index $\theta$ depends on $(d,\ga,p_{1},p_{2})$.
Given the constrain that $r_{p}>\frac{d}{p}+\ga>0$, we select $p_{1}$ and $p_{2}$ sufficiently close to $\frac{d}{-\ga}$ and apply the Sobolev inequality to obtain
\begin{align*}
\n{\wt{g}}_{L^{p_{1}}}\lesssim \n{\wt{g}}_{W^{r_{p},p}},\quad \n{\wt{g}}_{L^{p_{2}}}\lesssim \n{\wt{g}}_{W^{r_{p},p}},
\end{align*}
which together with \eqref{equ:Q-,bilinear estimate,f,g,step1} yields the desired estimate \eqref{equ:Q-,bilinear estimate,f,g}.
\end{proof}

The constraint in Theorem \ref{thm:main theorem} that $\ga\in [-\frac{d}{p_{0}},0]$ is partly due to the restriction $\frac{d}{p}+\ga>0$. Note that this restriction is a strict inequality, in contrast with Lemma \ref{lemma:Q+,bilinear estimate} for the gain term where it is permissible for $r_{p}\geq \frac{d}{p}+\ga\geq 0$. Consequently, for the space-time bilinear estimates for the loss term, Strichartz estimates \eqref{equ:strichartz,different frequency,p<p0} with the range $p\in [p_{0}-\ve,p_{0})$ are required to address this issue.

\begin{proof}[\textbf{Proof of Proposition $\ref{lemma:bilinear estimates,loss,gain}$ for the loss term}]
We provide the proof for estimate \eqref{equ:Q-,bilinear estimate,Hs,1}, as the proof for estimate \eqref{equ:Q-,bilinear estimate,Hs,2} follows by a similar argument.

By duality,
\eqref{equ:Q-,bilinear estimate,Hs,1} is equivalent to
\begin{align}\label{equ:Q-,bilinear estimate,L2,equivalent}
\int \lrs{\lra{\nabla_{\xi}}^{r}\wt{Q}^{-}(U(t)\wt{f}_{0},U(t)\wt{g}_{0})} h dx d\xi dt\lesssim
T^{\frac{1}{2}}\n{\wt{f}_{0}}_{H_{x}^{s}H_{\xi}^{r}}\n{\wt{g}_{0}}_{H_{x}^{s_{1}}H_{\xi}^{r}} \n{h}_{L_{t}^{\infty}H_{x}^{-s}L_{\xi}^{2}} .
\end{align}
We denote by $I$ the integral in \eqref{equ:Q-,bilinear estimate,L2,equivalent}.
Inserting a Littlewood-Paley decomposition gives that
\begin{align*}
I=\sum_{N_{1},N_{2},N_{3}} I_{N_{1},N_{2},N_{3}}
\end{align*}
where
\begin{align*}
I_{N_{1},N_{2},N_{3}}=& \int \wt{Q}^{-}(P_{N_{1}}^{x}\lra{\nabla_{\xi}}^{r}\wt{f},P_{N_{2}}^{x}\wt{g}) P_{N_{3}}^{x}h dx d\xi dt
\end{align*}
with $\wt{f}=U(t)\wt{f}_{0}$ and $\wt{g}=U(t)\wt{g}_{0}$.

Since $\wt{Q}^{-}$ commutes with $P_{N_{3}}^{x}$, the top two frequencies
are comparable in size. Hence, we divide the sum into three cases as follows

Case A. $N_{1}\sim N_{2}\geq N_{3}$.

Case B. $N_{1}\sim N_{3}\geq N_{2}$.

Case C. $N_{2}\sim N_{3}\geq N_{1}$.

\bigskip

\textbf{Case A. $N_{1}\sim N_{2}\geq N_{3}$.}

Let $I_{A}$ denote the integral restricted to the Case A.
By Cauchy-Schwarz,
\begin{align*}
I_{A}\lesssim
 \sum_{N_{1}\sim N_{2}\geq N_{3}}\int
\n{\wt{Q}^{-}(P_{N_{1}}^{x}\lra{\nabla_{\xi}}^{r}\wt{f}, P_{N_{2}}^{x}\wt{g})}_{L_{x,\xi}^{2}}
\n{P_{N_{3}}^{x}h}_{L_{x,\xi}^{2}}dt.
\end{align*}
By using the time-independent estimate \eqref{equ:Q-,bilinear estimate,f,g} in Lemma \ref{lemma:Q-,bilinear estimate} and then H\"{o}lder inequality, we have
\begin{align*}
I_{A} \lesssim & \sum_{N_{1}\sim N_{2}\geq N_{3}}
\int \bbn{\n{P_{N_{1}}^{x}\lra{\nabla_{\xi}}^{r}\wt{f}}_{L_{\xi}^{2}}
\n{P_{N_{2}}^{x}\widetilde{g}}_{W_{\xi}^{r_{p},p}}}_{L_{x}^{2}}
\n{P_{N_{3}}^{x}h}_{L_{x,\xi}^{2}}dt \\
\leq& \sum_{N_{1}\sim N_{2}\geq N_{3}}
 \int \n{P_{N_{1}}^{x}\lra{\nabla_{\xi}}^{r}\wt{f}}_{L_{x}^{2}L_{\xi}^{2}}
 \n{P_{N_{2}}^{x}\widetilde{g}}_{L_{x}^{\infty}W_{\xi}^{r_{p},p}}
 \n{P_{N_{3}}^{x}h}_{L_{x,\xi}^{2}} dt,\\
 \leq& \sum_{N_{1}\sim N_{2}\geq N_{3}}
 \n{P_{N_{1}}^{x}\lra{\nabla_{\xi}}^{r}\wt{f}}_{L_{t}^{\infty}L_{x}^{2}L_{\xi}^{2}}
 \n{P_{N_{2}}^{x}\widetilde{g}}_{L_{t}^{p}L_{x}^{\infty}W_{\xi}^{r_{p},p}}
 \n{P_{N_{3}}^{x}h}_{L_{t}^{2}L_{x,\xi}^{2}} ,
\end{align*}
where in the last inequality we have discarded the unimportant factor $T^{\frac{1}{2}-\frac{1}{p}}$.
By Bernstein inequality that
\begin{align*}
\n{P_{N_{2}}^{x}\lra{\nabla_{\xi}}^{r_{p}}\wt{g}}_{L_{x}^{\infty}}\lesssim N_{2}^{\frac{d}{p}} \n{P_{N_{2}}^{x}\lra{\nabla_{\xi}}^{r_{p}}\wt{g}}_{L_{x}^{p}},
\end{align*}
we obtain
\begin{align*}
I_{A}\lesssim& \sum_{N_{1}\sim N_{2}\geq N_{3}}N_{2}^{\frac{d}{p}}
 \n{P_{N_{1}}^{x}\lra{\nabla_{\xi}}^{r}\wt{f}}_{L_{t}^{\infty}L_{x}^{2}L_{\xi}^{2}}
\n{P_{N_{2}}^{x}\lra{\nabla_{\xi}}^{r_{p}}\wt{g}}_{L_{t,x,\xi}^{p}} \n{P_{N_{3}}^{x}h}_{L_{t}^{2}L_{x,\xi}^{2}},
\end{align*}
By Strichartz estimate \eqref{equ:strichartz estimate,regularity,p>2} in Corollary \ref{lemma:regularity strichartz} with $p<p_{0}$, we have
\begin{align*}
I_{A}
 \lesssim& \sum_{N_{1}\sim N_{2}\geq N_{3}}N_{2}^{\frac{d}{p}}N_{2}^{\frac{1}{p}+\ve}
 \n{P_{N_{1}}^{x}\lra{\nabla_{\xi}}^{r}\wt{f}}_{L_{t}^{\infty}L_{x}^{2}L_{\xi}^{2}}
\n{P_{N_{2}}^{x}\lra{\nabla_{\xi}}^{r_{p}+s_{p}+\frac{1}{p}}\wt{g}_{0}}_{L_{x,\xi}^{2}}
 \n{P_{N_{3}}^{x}h}_{L_{t}^{2}L_{x,\xi}^{2}}.
\end{align*}
Noticing that the range
\begin{align} \label{equ:range,parameters}
s>\frac{d}{2}-\frac{1}{p_{0}}=\frac{d+1}{p_{0}},\quad r>\frac{d}{2}+\ga,\quad s_{p}>\frac{d}{2}-\frac{d+1}{p},\quad r_{p}>\frac{d}{p}+\ga>0,
\end{align}
for fixed $(s,r)$, we can choose $p=p_{0}-\ve$ with $\ve$ sufficiently small and appropriate parameters $(s_{p},r_{p})$ such that
$$s>\frac{d}{p_{0}-\ve}+\frac{1}{p_{0}-\ve}+\ve, \quad r>r_{p}+s_{p}+\frac{1}{p}.$$
Hence, we arrive at
\begin{align*}
 I_{A}\lesssim&\sum_{N_{1}\sim N_{2}\geq N_{3}}N_{2}^{s}
 \n{P_{N_{1}}^{x}\lra{\nabla_{\xi}}^{r}\wt{f}_{0}}_{L_{x}^{2}L_{\xi}^{2}}
\n{P_{N_{2}}^{x}\lra{\nabla_{\xi}}^{r}\wt{g}_{0}}_{L_{x,\xi}^{2}}
 \n{P_{N_{3}}^{x}h}_{L_{t}^{2}L_{x,\xi}^{2}}.
\end{align*}
Using Bernstein inequality, carrying out the sum in $N_{3}$ and applying Cauchy-Schwarz inequality in $(N_{1},N_{2})$, we have
\begin{align}\label{equ:Q-,A,final estimate}
I_{A}\lesssim& \sum_{N_{1}\sim N_{2}\geq N_{3}}N_{1}^{-s}N_{2}^{s-s_{1}}N_{3}^{s_{1}}
 \n{P_{N_{1}}^{x}\wt{f}_{0}}_{H_{x}^{s}H_{\xi}^{r}}
\n{P_{N_{2}}^{x}\wt{g}_{0}}_{H_{x}^{s_{1}}H_{\xi}^{r}}
 \n{P_{N_{3}}^{x}h}_{L_{t}^{2}H_{x}^{-s_{1}}L_{\xi}^{2}}\\
 \lesssim&\sum_{N_{1}\sim N_{2}}N_{1}^{-s}N_{2}^{s}
 \n{P_{N_{1}}^{x}\wt{f}_{0}}_{H_{x}^{s}H_{\xi}^{r}}
\n{P_{N_{2}}^{x}\wt{g}_{0}}_{H_{x}^{s_{1}}H_{\xi}^{r}}
 \n{h}_{L_{t}^{2}H_{x}^{-s_{1}}L_{\xi}^{2}}\notag\\
 \lesssim& \n{\wt{f}_{0}}_{H_{x}^{s}H_{\xi}^{r}}
\n{\wt{g}_{0}}_{H_{x}^{s_{1}}H_{\xi}^{r}}
 \n{h}_{L_{t}^{2}H_{x}^{-s_{1}}L_{\xi}^{2}}\notag\\
\lesssim& T^{\frac{1}{2}}\n{\wt{f}_{0}}_{H_{x}^{s}H_{\xi}^{r}}
\n{\wt{g}_{0}}_{H_{x}^{s_{1}}H_{\xi}^{r}}
 \n{h}_{L_{t}^{\infty}H_{x}^{-s_{1}}L_{\xi}^{2}}.\notag
\end{align}
This completes the proof of \eqref{equ:Q-,bilinear estimate,L2,equivalent} for Case A.

\bigskip

\textbf{Case B. $N_{1}\sim N_{3}\geq N_{2}$.}

Using the same process as applied to estimate \eqref{equ:Q-,A,final estimate} in Case A, we arrive at
\begin{align*}
I_{B} \lesssim& \sum_{N_{1}\sim N_{3}\geq N_{2}}N_{1}^{-s}N_{2}^{s-s_{1}}N_{3}^{s_{1}}
 \n{P_{N_{1}}^{x}\wt{f}_{0}}_{H_{x}^{s}H_{\xi}^{r}}
\n{P_{N_{2}}^{x}\wt{g}_{0}}_{H_{x}^{s_{1}}H_{\xi}^{r}}
 \n{P_{N_{3}}^{x}h}_{L_{t}^{2}H_{x}^{-s_{1}}L_{\xi}^{2}}\\
 \lesssim &\sum_{N_{1}\sim N_{3}}N_{1}^{-s}N_{3}^{s}
 \n{P_{N_{1}}^{x}\wt{f}_{0}}_{H_{x}^{s}H_{\xi}^{r}}
\n{\wt{g}_{0}}_{H_{x}^{s_{1}}H_{\xi}^{r}}
 \n{P_{N_{3}}^{x}h}_{L_{t}^{2}H_{x}^{-s_{1}}L_{\xi}^{2}}\\
 \lesssim&\n{\wt{f}_{0}}_{H_{x}^{s}H_{\xi}^{r}}
\n{\wt{g}_{0}}_{H_{x}^{s_{1}}H_{\xi}^{r}}
 \n{h}_{L_{t}^{2}H_{x}^{-s_{1}}L_{\xi}^{2}}\\
 \lesssim& T^{\frac{1}{2}}\n{\wt{f}_{0}}_{H_{x}^{s}H_{\xi}^{r}}
\n{\wt{g}_{0}}_{H_{x}^{s_{1}}H_{\xi}^{r}}
 \n{h}_{L_{t}^{\infty}H_{x}^{-s_{1}}L_{\xi}^{2}},
\end{align*}
which completes the proof of \eqref{equ:Q-,bilinear estimate,L2,equivalent} for Case B.

\bigskip
\textbf{Case C. $N_{2}\sim N_{3}\geq N_{1}$.}

It should be noted that the strategy employed in Cases A and B is not applicable in the current Case C, for which more detailed frequency analysis is required.
Let $I_{C}$ denote the integral restricted to the Case C.
Let $k_{1}$, $k_{2}$, and $k_{3}$ denote the Fourier coefficients corresponding to the
spatial variables $x_{1}$, $x_{2}$, and $x_{3}$, respectively. Decompose the $N_{2}$ and $N_{3}$ dyadic spaces
into $N_{1}$ size cubes. Due to the frequency constraint $k_{1}+k_{2}+k_{3}=0$, for each
choice $D$ of an $N_{1}$-cube within the $k_{2}$ space, the variable $k_{3}$ is
constrained to at most $10^{d}$ of $N_{1}$-cubes dividing the $N_{3}$ dyadic space.
For convenience, we will refer to these $10^{d}$ cubes as a single cube $D_{c}$ that
corresponds to $D$. Then we have
\begin{align*}
I_{C}=
\sum_{N_{2}\sim N_{3}\geq N_{1}}\sum_{D}
\int  \wt{Q}^{-}(P_{N_{1}}^{x}\lra{\nabla_{\xi}}^{r}\wt{f},P_{D}^{x}P_{N_{2}}^{x} \wt{g})
P_{D^{c}}^{x}P_{N_{3}}^{x}h dtdxd\xi.
\end{align*}
By Cauchy-Schwarz,
\begin{align*}
I_{C}\lesssim
\int \sum_{N_{2}\sim N_{3}\geq N_{1}}
\n{\wt{Q}^{-}(P_{N_{1}}^{x}\lra{\nabla_{\xi}}^{r}\wt{f},P_{D}^{x} P_{N_{2}}^{x}\wt{g})}_{l_{D}^{2}L_{x,\xi}^{2}}
\n{P_{D^{c}}^{x}P_{N_{3}}^{x}h}_{l_{D}^{2}L_{x,\xi}^{2}}dt.
\end{align*}
By using estimate \eqref{equ:Q-,bilinear estimate,f,g} in Lemma \ref{lemma:Q-,bilinear estimate} and then H\"{o}lder inequality, we have
\begin{align*}
I_{C} \lesssim &\int \sum_{N_{2}\sim N_{3}\geq N_{1}}
 \bbn{\n{P_{N_{1}}^{x}\lra{\nabla_{\xi}}^{r}\wt{f}}_{L_{\xi}^{2}}
\n{P_{D}^{x}P_{N_{2}}^{x}\widetilde{g}}_{W_{\xi}^{r_{p},p}}}_{l_{D}^{2}L_{x}^{2}}
\n{P_{D^{c}}^{x}P_{N_{3}}^{x}h}_{l_{D}^{2}L_{x,\xi}^{2}}dt \\
\leq& \sum_{N_{2}\sim N_{3}\geq N_{1}}
 \n{P_{N_{1}}^{x}\lra{\nabla_{\xi}}^{r}\wt{f}}_{L_{t}^{\infty}L_{x}^{2}L_{\xi}^{2}}
 \n{P_{D}^{x}P_{N_{2}}^{x}\widetilde{g}}_{L_{t}^{p}l_{D}^{2}L_{x}^{\infty}W_{\xi}^{r_{p},p}}
 \n{P_{D^{c}}^{x}P_{N_{3}}^{x}h}_{L_{t}^{2}l_{D}^{2}L_{x,\xi}^{2}} .
\end{align*}
By Minkowski inequality, Bernstein inequality, and Strichartz estimate \eqref{equ:strichartz estimate,regularity,p>2} with noncentered frequency localization,
we obtain
\begin{align*}
I_{C}\lesssim& \sum_{N_{2}\sim N_{3}\geq N_{1}}N_{1}^{\frac{d}{p}}
 \n{P_{N_{1}}^{x}\lra{\nabla_{\xi}}^{r}\wt{f}}_{L_{t}^{\infty}L_{x}^{2}L_{\xi}^{2}}
\n{P_{D}^{x}P_{N_{2}}^{x}\lra{\nabla_{\xi}}^{r_{p}}\wt{g}}_{l_{D}^{2}L_{t,x,\xi}^{p}}
 \n{P_{N_{3}}^{x}h}_{L_{t}^{2}L_{x,\xi}^{2}}\\
 \lesssim& \sum_{N_{2}\sim N_{3}\geq N_{1}}N_{1}^{\frac{d}{p}}N_{1}^{\frac{1}{p}+\ve}
 \n{P_{N_{1}}^{x}\lra{\nabla_{\xi}}^{r}\wt{f}}_{L_{t}^{\infty}L_{x}^{2}L_{\xi}^{2}}
\n{P_{D}^{x}P_{N_{2}}^{x}\lra{\nabla_{\xi}}^{r_{p}+s_{p}+\frac{1}{p}}\wt{g}_{0}}_{l_{D}^{2}L_{x,\xi}^{2}}
 \n{P_{N_{3}}^{x}h}_{L_{t}^{2}L_{x,\xi}^{2}}.
 \end{align*}
In the same way as \eqref{equ:range,parameters}, by choosing appropriate parameters $(p,s_{p},r_{p})$ such that
$$s>s_{p}+\frac{d}{p}, \quad r>r_{p}+s_{p}+\frac{1}{p},$$
we arrive at
 \begin{align*}
 I_{C}\lesssim &\sum_{N_{2}\sim N_{3}} \n{\lra{\nabla_{\xi}}^{r}\wt{f}_{0}}_{H_{x}^{s}L_{\xi}^{2}}
 \n{P_{N_{2}}^{x}\lra{\nabla_{\xi}}^{r}\wt{g}_{0}}_{L_{x,\xi}^{2}}
 \n{P_{N_{3}}^{x}h}_{L_{t}^{2}L_{x,\xi}^{2}}\\
 \lesssim&\n{\wt{f}_{0}}_{H_{x}^{s}H_{\xi}^{r}}
\n{\wt{g}_{0}}_{H_{x}^{s_{1}}H_{\xi}^{r}}
 \n{h}_{L_{t}^{2}H_{x}^{-s_{1}}L_{\xi}^{2}}\\
 \lesssim& T^{\frac{1}{2}}\n{\wt{f}_{0}}_{H_{x}^{s}H_{\xi}^{r}}
\n{\wt{g}_{0}}_{H_{x}^{s_{1}}H_{\xi}^{r}}
 \n{h}_{L_{t}^{\infty}H_{x}^{-s_{1}}L_{\xi}^{2}},
\end{align*}
which completes the proof of \eqref{equ:Q-,bilinear estimate,L2,equivalent} for Case C.
 \end{proof}

 \subsection{Bilinear Estimates for the Gain Term}\label{sec:Bilinear Estimates for the Gain Term}
 The goal of the section is to prove the bilinear estimates for the gain term
 \begin{align}
\n{\wt{Q}^{+}(U(t)\tilde{f}_{0}, U(t)\tilde{g}_{0})}_{L_{t}^{1} H_{x}^{s_{1}} H_{\xi}^{r}} \lesssim T^{\frac{1}{2}}\n{\wt{f}_{0}}_{H_{x}^{s_{1}}H_{\xi}^{r}}
\n{\wt{g}_{0}}_{H_{x}^{s}H_{\xi}^{r}}\label{equ:Q+,bilinear estimate,Hs,1}\\
\n{\wt{Q}^{+}(U(t)\tilde{f}_{0}, U(t)\tilde{g}_{0})}_{L_{t}^{1} H_{x}^{s_{1}} H_{\xi}^{r}} \lesssim T^{\frac{1}{2}}\n{\wt{f}_{0}}_{H_{x}^{s}H_{\xi}^{r}}
\n{\wt{g}_{0}}_{H_{x}^{s_{1}}H_{\xi}^{r}}\label{equ:Q+,bilinear estimate,Hs,2}
\end{align}

 Before proving \eqref{equ:Q+,bilinear estimate,Hs,1}--\eqref{equ:Q+,bilinear estimate,Hs,2}, we first provide time-independent estimates as follows.
\begin{lemma}\label{lemma:Q+,bilinear estimate}
Let $r_{p}\geq \frac{d}{p}+\ga\geq 0$. Then we have
  \begin{align}
\n{\wt{Q}^{+}(\wt{f},\wt{g})}_{L_{\xi}^{2}}
\lesssim& \n{\wt{f}}_{L_{\xi}^{2}}
\n{\wt{g}}_{W_{\xi}^{r_{p},p}}, \label{equ:Q+,bilinear estimate,PM,f,g}\\
\n{\wt{Q}^{+}(\wt{f},\wt{g})}_{L_{\xi}^{2}}
\lesssim&  \n{\wt{f}}_{W_{\xi}^{r_{p},p}}
\n{\wt{g}}_{L_{\xi}^{2}}.\label{equ:Q+,bilinear estimate,PM,g,f}
\end{align}
\end{lemma}
For the proof of Lemma \ref{lemma:Q+,bilinear estimate}, see for example \cite[Lemma 2.2]{CSZ24well}. Because the collision operator is $x$-independent, the collision estimates
\eqref{equ:Q+,bilinear estimate,PM,f,g}--\eqref{equ:Q+,bilinear estimate,PM,g,f} apply equally to both the $\R^{d}$ and $\T^{d}$ spatial domains.

\begin{proof}[\textbf{Proof of Proposition $\ref{lemma:bilinear estimates,loss,gain}$ for the gain term}]
We only need to prove estimate \eqref{equ:Q+,bilinear estimate,Hs,1}, as the proof for estimate \eqref{equ:Q+,bilinear estimate,Hs,2} follows by a similar argument.

By duality,
\eqref{equ:Q+,bilinear estimate,Hs,1} is equivalent to
\begin{align}\label{equ:Q+,bilinear estimate,L2,equivalent}
\int \wt{Q}^{+}(U(t)\wt{f}_{0},U(t)\wt{g}_{0}) h dx d\xi dt\lesssim
T^{\frac{1}{2}}\n{\wt{f}_{0}}_{H_{x}^{s}H_{\xi}^{r}}\n{\wt{g}_{0}}_{H_{x}^{s_{1}}H_{\xi}^{r}} \n{h}_{L_{t}^{\infty}H_{x}^{-s}H_{\xi}^{-r}} .
\end{align}
We denote by $I$ the integral in \eqref{equ:Q+,bilinear estimate,L2,equivalent}.
Inserting a Littlewood-Paley decomposition gives that
\begin{align*}
I=\sum_{\substack{M_{1},M_{2},M_{3}\\N_{1},N_{2},N_{3}}} I_{M_{1},M_{2},M_{3},N_{1},N_{2},N_{3}}
\end{align*}
where
\begin{align*}
I_{M_{1},M_{2},M_{3},N_{1},N_{2},N_{3}}= \int \wt{Q}^{+}(P_{N_{1}}^{x}P_{M_{1}}^{\xi}\wt{f},P_{N_{2}}^{x}P_{M_{2}}^{\xi}\wt{g}) P_{N_{3}}^{x}P_{M}^{\xi}h dx d\xi dt.
\end{align*}
with $\wt{f}=U(t)\wt{f}_{0}$ and $\wt{g}=U(t)\wt{g}_{0}$.

Since $\wt{Q}^{+}$ also commutes with $P_{N_{3}}^{x}$, the top two frequencies
are comparable in size, which yields three cases
\begin{equation*}
\left\{
\begin{aligned}
N_{1}\sim N_{2}\geq N_{3},\\
N_{1}\sim N_{3}\geq N_{2},\\
N_{2}\sim N_{3}\geq N_{1}.
\end{aligned}
\right.
\end{equation*}
In addition, one key observation, as seen in \cite{chen2023sharp,CSZ24well}, is that a similar property is hinted in the $\xi$-variable, that is,
\begin{align}\label{equ:property,constraint,projector,v,variable}
P_{M}^{\xi}\wt{Q}^{+}(P_{M_{1}}^{\xi}\wt{f},P_{M_{2}}^{\xi}\wt{g})=0, \quad \text{if $M\geq 10\max\lrs{M_{1},M_{2}}$.}
\end{align}
Indeed, note that
\begin{align*}
\cf_{\xi}\lrs{P_{M}^{\xi}\wt{Q}^{+}(P_{M_{1}}^{\xi}\wt{f},P_{M_{2}}^{\xi}\wt{g})}
=
\vp_{M}(v)\int_{\mathbb{S}^{2}}\int_{\mathbb{R}^{3}} (\vp_{M_{1}}f)(v^{\ast })(\vp_{M_{2}}g)(u^{\ast}) B(u-v,\omega)dud\omega.
\end{align*}
By the energy conservation that $|v|^{2}+|u|^{2}=|v^{*}|^{2}+|u^{*}|^{2}$, we have $|v|\leq |v^{*}|+|u^{*}|$ for all $(u,\omega)\in \R^{d}\times \mathbb{S}^{d-1}$. Together with $M\geq 10\max\lrs{M_{1},M_{2}}$, this lower bound implies that
\begin{align*}
\vp_{M}(v)\vp_{M_{1}}(v^{*})\vp_{M_{2}}(u^{*})=0, \quad \text{for all $(u,\omega)\in \R^{d}\times \mathbb{S}^{d-1}$}.
\end{align*}
Thus, we arrive at \eqref{equ:property,constraint,projector,v,variable} and the constraint that $M\lesssim \max\lrs{M_{1},M_{2}}$.

 Now, we divide the sum into six cases as follows

Case $A_{1}$. $N_{1}\sim N_{2}\geq N_{3}$, $M_{1}\geq M_{2}$. $\quad$ Case $A_{2}$. $N_{1}\sim N_{2}\geq N_{3}$, $M_{1}\leq M_{2}$.

Case $B_{1}$. $N_{1}\sim N_{3}\geq N_{2}$, $M_{1}\geq M_{2}$. $\quad$ Case $B_{2}$. $N_{1}\sim N_{3}\geq N_{2}$, $M_{1}\leq M_{2}$.

Case $C_{1}$. $N_{2}\sim N_{3}\geq N_{1}$, $M_{1}\geq M_{2}$. $\quad$ Case $C_{2}$. $N_{2}\sim N_{3}\geq N_{1}$, $M_{1}\leq M_{2}$.

Given the  symmetric estimates in Lemma \ref{lemma:Q+,bilinear estimate}, it is sufficient to treat Cases $A_{1}$, $B_{1}$, $C_{1}$. The Cases $A_{2}$, $B_{2}$, $C_{2}$ will follow similarly.

\vspace{1em}
\textbf{Case $A_{1}$. $N_{1}\sim N_{2}\geq N_{3}$, $M_{1}\geq M_{2}$.}

Let $I_{A_{1}}$ denote the integral restricted to the Case $A_{1}$. By Cauchy-Schwarz,
\begin{align*}
I_{A_{1}}\lesssim
\sum_{\substack{ _{\substack{ N_{1}\sim N_{2}\geq N_{3}  \\ M_{1}\geq M_{2},M_{1}\geq M_{3}}}}}
\int \n{\wt{Q}^{+}(P_{N_{1}}^{x}P_{M_{1}}^{\xi}\wt{f},P_{N_{2}}^{x}P_{M_{2}}^{\xi}\wt{g})}_{L_{x,\xi}^{2}}
\n{P_{N_{3}}^{x}P_{M_3}^{\xi}h}_{L_{x,\xi}^{2}}dt.
\end{align*}
By using estimate \eqref{equ:Q+,bilinear estimate,PM,f,g} in Lemma \ref{lemma:Q+,bilinear estimate} and then H\"{o}lder inequality, we have
\begin{align*}
I_{A_{1}} \lesssim &\sum_{\substack{ _{\substack{ N_{1}\sim N_{2}\geq N_{3}  \\ M_{1}\geq M_{2},M_{1}\geq M_{3}}}}}
\int \bbn{\n{P_{N_{1}}^{x}P_{M_{1}}^{\xi}\widetilde{f}}_{L_{\xi}^{2}}
\n{P_{N_{2}}^{x}P_{M_{2}}^{\xi }\widetilde{g}}_{W_{\xi}^{r_{p},p}}}_{L_{x}^{2}}
\n{P_{N_{3}}^{x}P_{M_{3}}^{\xi}h}_{L_{x,\xi}^{2}} dt\\
\leq&\sum_{\substack{ _{\substack{ N_{1}\sim N_{2}\geq N_{3}  \\ M_{1}\geq M_{2},M_{1}\geq M_{3}}}}}
 \int \n{ P_{N_{1}}^{x}P_{M_{1}}^{\xi }\widetilde{f}}_{L_{x}^{2}L_{\xi }^{2}}
 \n{ P_{N_{2}}^{x}P_{M_{2}}^{\xi }\widetilde{g}}_{L_{x}^{\infty}W_{\xi}^{r_{p},p}}
 \n{P_{N_{3}}^{x}P_{M_{3}}^{\xi}h}_{L_{x,\xi}^{2}}dt \\
\leq&\sum_{\substack{ _{\substack{ N_{1}\sim N_{2}\geq N_{3}  \\ M_{1}\geq M_{2},M_{1}\geq M_{3}}}}}
 \n{ P_{N_{1}}^{x}P_{M_{1}}^{\xi }\widetilde{f}}_{L_{t}^{\infty}L_{x}^{2}L_{\xi }^{2}}
 \n{ P_{N_{2}}^{x}P_{M_{2}}^{\xi }\widetilde{g}}_{L_{t}^{p}L_{x}^{\infty}W_{\xi}^{r_{p},p}}
 \n{P_{N_{3}}^{x}P_{M_{3}}^{\xi}h}_{L_{t}^{2}L_{x,\xi}^{2}} .
\end{align*}
where in the last inequality we have discarded the unimportant factor $T^{\frac{1}{2}-\frac{1}{p}}$.
By Bernstein inequality that
\begin{align*}
\n{P_{N_{2}}^{x}\lra{\nabla_{\xi}}^{r_{p}}\wt{g}}_{L_{x}^{\infty}}\lesssim N_{2}^{\frac{d}{p}} \n{P_{N_{2}}^{x}\lra{\nabla_{\xi}}^{r_{p}}\wt{g}}_{L_{x}^{p}},
\end{align*}
we obtain
\begin{align*}
I_{A_{1}} \lesssim &\sum_{\substack{ _{\substack{ N_{1}\sim N_{2}\geq N_{3}  \\ M_{1}\geq M_{2},M_{1}\geq M_{3}}}}}
N_{2}^{\frac{d}{p}}
\n{P_{N_{1}}^{x}P_{M_{1}}^{\xi}\widetilde{f}}_{L_{t}^{\infty}L_{x,\xi}^{2}}
 \n{P_{N_{2}}^{x}P_{M_{2}}^{\xi}\lra{\nabla_{\xi}}^{r_{p}} \widetilde{g}}_{L_{t}^{p}L_{x}^{p}L_{\xi}^{p}}
\n{P_{N_{3}}^{x}P_{M_{3}}^{\xi}h}_{L_{t}^{2}L_{x,\xi}^{2}}.
\end{align*}
By Strichartz estimate \eqref{equ:strichartz estimate,regularity} in Corollary \ref{lemma:regularity strichartz}, for $p\geq \frac{2(d+2)}{d}$ we have
\begin{align*}
I_{A_{1}}\lesssim&\sum_{\substack{ _{\substack{ N_{1}\sim N_{2}\geq N_{3}  \\ M_{1}\geq M_{2},M_{1}\geq M_{3}}}}}
N_{2}^{\frac{d}{p}+s_{p}}
\n{P_{N_{1}}^{x}P_{M_{1}}^{\xi}\widetilde{f}_{0}}_{L_{x}^{2}L_{\xi}^{2}}
 \n{P_{N_{2}}^{x}P_{M_{2}}^{\xi}\widetilde{g}_{0}}_{L_{x}^{2}H_{\xi}^{r_{p}+s_{p}+\frac{1}{p}}}
\n{P_{N_{3}}^{x}P_{M_{3}}^{\xi}h}_{L_{t}^{2}L_{x}^{2}L_{\xi}^{2}}.
\end{align*}
By choosing appropriate parameters $(p,s_{p},r_{p})$ such that
$$s>s_{p}+\frac{d}{p}, \quad r>r_{p}+s_{p}+\frac{1}{p},$$
we carry out the sum in $M_{2}$ to obtain
\begin{align*}
I_{A_{1}}\lesssim&\sum_{\substack{ _{\substack{ N_{1}\sim N_{2}\geq N_{3}  \\ M_{1}\geq M_{3}}}}}
N_{2}^{\frac{d}{p}+s_{p}}
\n{P_{N_{1}}^{x}P_{M_{1}}^{\xi}\widetilde{f}_{0}}_{L_{x}^{2}L_{\xi}^{2}}
 \n{P_{N_{2}}^{x}\widetilde{g}_{0}}_{L_{x}^{2}H_{\xi}^{r}}
\n{P_{N_{3}}^{x}P_{M_{3}}^{\xi}h}_{L_{t}^{2}L_{x}^{2}L_{\xi}^{2}}\\
\lesssim &\sum_{\substack{ _{\substack{ N_{1}\sim N_{2}\geq N_{3}  \\ M_{1}\geq M_{3}}}}}N_{2}^{s}
\n{P_{N_{1}}^{x}P_{M_{1}}^{\xi}\widetilde{f}_{0}}_{L_{x}^{2}L_{\xi}^{2}}
 \n{P_{N_{2}}^{x}\widetilde{g}_{0}}_{L_{x}^{2}H_{\xi}^{r}}
\n{P_{N_{3}}^{x}P_{M_{3}}^{\xi}h}_{L_{t}^{2}L_{x}^{2}L_{\xi}^{2}}.
   \end{align*}
Using Bernstein inequality again, we obtain
\begin{align}\label{equ:Q+,A1,final estimate}
I_{A_{1}}\lesssim \sum_{\substack{ _{\substack{ N_{1}\sim N_{2}\geq N_{3}  \\ M_{1}\geq M_{3}}}}}
\frac{N_{2}^{s-s_{1}}N_{3}^{s_{1}}M_{3}^{r}}{N_{1}^{s}M_{1}^{r}}
\n{P_{N_{1}}^{x}P_{M_{1}}^{\xi}\widetilde{f}_{0}}_{H_{x}^{s}H_{\xi}^{r}}
 \n{P_{N_{2}}^{x}\widetilde{g}_{0}}_{H_{x}^{s_{1}}H_{\xi}^{r}}
\n{P_{N_{3}}^{x}P_{M_{3}}^{\xi}h}_{L_{t}^{2}H_{x}^{-s_{1}}H_{\xi}^{-r}}.
\end{align}
Carrying out the sum in $N_{3}$ and using Cauchy-Schwarz in $(N_{1},N_{2})$ and $(M_{1},M_{3})$, we have
\begin{align*}
I_{A_{1}}\lesssim& \sum_{\substack{ _{\substack{ N_{1}\sim N_{2} \\ M_{1}\geq M_{3}}}}}
\frac{N_{2}^{s}M_{3}^{r}}{N_{1}^{s}M_{1}^{r}}
\n{P_{N_{1}}^{x}P_{M_{1}}^{\xi}\widetilde{f}_{0}}_{H_{x}^{s}H_{\xi}^{r}}
 \n{P_{N_{2}}^{x}\widetilde{g}_{0}}_{H_{x}^{s_{1}}H_{\xi}^{r}}
\n{P_{M_{3}}^{\xi}h}_{L_{t}^{2}H_{x}^{-s_{1}}H_{\xi}^{-r}}\\
\lesssim& \n{\wt{f}_{0}}_{H_{x}^{s}H_{\xi}^{r}}\n{\wt{g}_{0}}_{H_{x}^{s_{1}}H_{\xi}^{r}}
\n{h}_{L_{t}^{2}H_{x}^{-s_{1}}H_{\xi}^{-r}}\\
\lesssim& T^{\frac{1}{2}}\n{\wt{f}_{0}}_{H_{x}^{s}H_{\xi}^{r}}\n{\wt{g}_{0}}_{H_{x}^{s_{1}}H_{\xi}^{r}}
\n{h}_{L_{t}^{\infty}H_{x}^{-s_{1}}H_{\xi}^{-r}}.
\end{align*}
which completes the proof of \eqref{equ:Q+,bilinear estimate,L2,equivalent} for Case $A_{1}$.

\vspace{1em}
\textbf{Case $B_{1}$. $N_{1}\sim N_{3}\geq N_{2}$, $M_{1}\geq M_{2}$.}

Following the same process as for estimate \eqref{equ:Q+,A1,final estimate} in Case $A_{1}$, we arrive at
\begin{align*}
I_{B_{1}}\lesssim \sum_{\substack{ _{\substack{ N_{1}\sim N_{3}\geq N_{2}  \\ M_{1}\geq M_{3}}}}}
\frac{N_{2}^{s-s_{1}}N_{3}^{s_{1}}M_{3}^{r}}{N_{1}^{s}M_{1}^{r}}
\n{P_{N_{1}}^{x}P_{M_{1}}^{\xi}\widetilde{f}_{0}}_{H_{x}^{s}H_{\xi}^{r}}
 \n{P_{N_{2}}^{x}\widetilde{g}_{0}}_{H_{x}^{s_{1}}H_{\xi}^{r}}
\n{P_{N_{3}}^{x}P_{M_{3}}^{\xi}h}_{L_{t}^{2}H_{x}^{-s_{1}}H_{\xi}^{-r}}.
\end{align*}
Due to that $s_{1}\leq s$, we first carry out the sum in $N_{2}$ and then use Cauchy-Schwarz inequality to obtain
\begin{align*}
I_{B_{1}}\lesssim& \sum_{\substack{ _{\substack{ N_{1}\sim N_{3} \\ M_{1}\geq M_{3}}}}}
\frac{N_{3}^{s}M_{3}^{r}}{N_{1}^{s}M_{1}^{r}}
\n{P_{N_{1}}^{x}P_{M_{1}}^{\xi}\widetilde{f}_{0}}_{H_{x}^{s}H_{\xi}^{r}}
 \n{\widetilde{g}_{0}}_{H_{x}^{s_{1}}H_{\xi}^{r}}
\n{P_{N_{3}}^{x}P_{M_{3}}^{\xi}h}_{L_{t}^{2}H_{x}^{-s_{1}}H_{\xi}^{-r}}\\
\lesssim& \n{\wt{f}_{0}}_{H_{x}^{s}H_{\xi}^{r}}\n{\wt{g}_{0}}_{H_{x}^{s_{1}}H_{\xi}^{r}}
\n{h}_{L_{t}^{2}H_{x}^{-s_{1}}H_{\xi}^{-r}}\\
\lesssim& T^{\frac{1}{2}}\n{\wt{f}_{0}}_{H_{x}^{s}H_{\xi}^{r}}\n{\wt{g}_{0}}_{H_{x}^{s_{1}}H_{\xi}^{r}}
\n{h}_{L_{t}^{\infty}H_{x}^{-s_{1}}H_{\xi}^{-r}},
\end{align*}
which completes the proof of \eqref{equ:Q+,bilinear estimate,L2,equivalent} for Case $B_{1}$.

\bigskip
\textbf{Case $C_{1}$. $N_{2}\sim N_{3}\geq N_{1}$, $M_{1}\geq M_{2}$.}

Let $I_{C_{1}}$ denote the integral restricted to the Case $C_{1}$.
 Decompose the $N_{2}$ and $N_{3}$ dyadic spaces
into $N_{1}$ size cubes. In a manner analogous to the analysis for the loss term in Case $C$, for each
choice $D$ of an $N_{1}$-cube within the $k_{2}$ space, the variable $k_{3}$ is
constrained to at most $10^{d}$ of $N_{1}$-cubes dividing the $N_{3}$ dyadic space.
For convenience, we also refer to these $10^{d}$ cubes as a single cube $D_{c}$ that
corresponds to $D$. Then we have
\begin{align*}
I_{C_{1}}=\sum_{\substack{ _{\substack{ N_{2}\sim N_{3}\geq N_{1}  \\ M_{1}\geq M_{3}}}}}\sum_{D}
\int \wt{Q}^{+}(P_{N_{1}}^{x}P_{M_{1}}^{\xi}\wt{f},P_{N_{2}}^{x}P_{M_{2}}^{\xi}P_{D}^{x}\wt{g})
 P_{N_{3}}^{x}P_{M_{3}}^{\xi}P_{D^{c}}^{x}h dx d\xi dt.
\end{align*}
By Cauchy-Schwarz, we get
\begin{align*}
I_{C_{1}}\lesssim
\int \sum_{\substack{ _{\substack{ N_{2}\sim N_{3}\geq N_{1}  \\ M_{1}\geq M_{3}}}}}
\n{\wt{Q}^{+}(P_{N_{1}}^{x}P_{M_{1}}^{\xi}\wt{f},P_{N_{2}}^{x}P_{M_{2}}^{\xi}P_{D}^{x}\wt{g})}_{l_{D}^{2}L_{x,\xi}^{2}}
\n{P_{N_{3}}^{x}P_{M_{3}}^{\xi}P_{D^{c}}^{x}h}_{l_{D}^{2}L_{x,\xi}^{2}}dt.
\end{align*}
By using estimate \eqref{equ:Q+,bilinear estimate,PM,f,g} in Lemma \ref{lemma:Q+,bilinear estimate} and then H\"{o}lder inequality, we have
\begin{align*}
I_{C_{1}} \lesssim &\int \sum_{\substack{ _{\substack{ N_{2}\sim N_{3}\geq N_{1}  \\ M_{1}\geq M_{3}}}}}
 \bbn{\n{P_{N_{1}}^{x}P_{M_{1}}^{\xi}\wt{f}}_{L_{\xi}^{2}}
\n{P_{N_{2}}^{x}P_{M_{2}}^{\xi}P_{D}^{x}\wt{g}}_{W_{\xi}^{r_{p},p}}}_{l_{D}^{2}L_{x}^{2}}
\n{P_{N_{3}}^{x}P_{M_{3}}^{\xi}P_{D^{c}}^{x}h}_{l_{D}^{2}L_{x,\xi}^{2}}dt \\
\leq& \sum_{\substack{ _{\substack{ N_{2}\sim N_{3}\geq N_{1}  \\ M_{1}\geq M_{3}}}}}
 \n{P_{N_{1}}^{x}P_{M_{1}}^{\xi}\wt{f}}_{L_{t}^{\infty}L_{x}^{2}L_{\xi}^{2}}
 \n{P_{N_{2}}^{x}P_{M_{2}}^{\xi}P_{D}^{x}\wt{g}}_{L_{t}^{p}l_{D}^{2}L_{x}^{\infty}W_{\xi}^{r_{p},p}}
 \n{P_{N_{3}}^{x}P_{M_{3}}^{\xi}P_{D^{c}}^{x}h}_{L_{t}^{2}l_{D}^{2}L_{x,\xi}^{2}} .
\end{align*}
By Minkowski inequality, Bernstein inequality, and Strichartz estimate \eqref{equ:strichartz estimate,regularity} with noncentered frequency localization,
we obtain
\begin{align*}
I_{C_{1}}\lesssim& \sum_{\substack{ _{\substack{ N_{2}\sim N_{3}\geq N_{1}  \\ M_{1}\geq M_{3}}}}}N_{1}^{\frac{d}{p}}
 \n{P_{N_{1}}^{x}P_{M_{1}}^{\xi}\wt{f}}_{L_{t}^{\infty}L_{x}^{2}L_{\xi}^{2}}
\n{P_{N_{2}}^{x}P_{M_{2}}^{\xi}P_{D}^{x}\lra{\nabla_{\xi}}^{r_{p}}\wt{g}}_{l_{D}^{2}L_{t,x,\xi}^{p}}
 \n{P_{N_{3}}^{x}P_{M_{3}}^{\xi}P_{D^{c}}^{x}h}_{l_{D}^{2}L_{x,\xi}^{2}}\\
 \lesssim& \sum_{\substack{ _{\substack{ N_{2}\sim N_{3}\geq N_{1}  \\ M_{1}\geq M_{3}}}}}
 N_{1}^{s_{p}+\frac{d}{p}}
 \n{P_{N_{1}}^{x}P_{M_{1}}^{\xi}\wt{f}_{0}}_{L_{x}^{2}L_{\xi}^{2}}
\n{P_{N_{2}}^{x}P_{M_{2}}^{\xi}\wt{g}_{0}}_{L_{x}^{2}H_{\xi}^{r_{p}+s_{p}+\frac{1}{p}}}
 \n{P_{N_{3}}^{x}P_{M_{3}}^{\xi}h}_{L_{t}^{2}L_{x,\xi}^{2}}\\
 \lesssim& \sum_{\substack{ _{\substack{ N_{2}\sim N_{3}\geq N_{1}  \\ M_{1}\geq M_{3}}}}}
 N_{1}^{s_{p}+\frac{d}{p}}\frac{M_{3}^{r}}{M_{1}^{r}}
 \n{P_{N_{1}}^{x}P_{M_{1}}^{\xi}\wt{f}_{0}}_{L_{x}^{2}H_{\xi}^{r}}
\n{P_{N_{2}}^{x}P_{M_{2}}^{\xi}\wt{g}_{0}}_{L_{x}^{2}H_{\xi}^{r_{p}+s_{p}+\frac{1}{p}}}
 \n{P_{N_{3}}^{x}P_{M_{3}}^{\xi}h}_{L_{t}^{2}L_{x}^{2}H_{\xi}^{-r}}\\
 \lesssim &\sum_{\substack{ _{\substack{ N_{2}\sim N_{3} \\ M_{1}\geq M_{3}}}}}
 \frac{M_{3}^{r}}{M_{1}^{r}}
 \n{P_{M_{1}}^{\xi}\wt{f}_{0}}_{H_{x}^{s}L_{\xi}^{2}}
 \n{P_{N_{2}}^{x}P_{M_{2}}^{\xi}\wt{g}_{0}}_{L_{x}^{2}H_{\xi}^{r}}
 \n{P_{N_{3}}^{x}P_{M_{3}}^{\xi}h}_{L_{t}^{2}L_{x}^{2}H_{\xi}^{-r}}.
 \end{align*}
 where in the last inequality we have chosen appropriate parameters $(p,s_{p},r_{p})$ such that
$$s>s_{p}+\frac{d}{p}, \quad r>r_{p}+s_{p}+\frac{1}{p}.$$
By Cauchy-Schwarz, we arrive at
\begin{align*}
 I_{C_{1}}
 \lesssim& \n{\wt{f}_{0}}_{H_{x}^{s}H_{\xi}^{r}}\n{\wt{g}_{0}}_{H_{x}^{s_{1}}H_{\xi}^{r}}
\n{h}_{L_{t}^{2}H_{x}^{-s_{1}}H_{\xi}^{-r}}\\
\lesssim& T^{\frac{1}{2}}\n{\wt{f}_{0}}_{H_{x}^{s}H_{\xi}^{r}}\n{\wt{g}_{0}}_{H_{x}^{s_{1}}H_{\xi}^{r}}
\n{h}_{L_{t}^{\infty}H_{x}^{-s_{1}}H_{\xi}^{-r}},
\end{align*}
which completes the proof of \eqref{equ:Q+,bilinear estimate,L2,equivalent} for Case $C_{1}$.
\end{proof}

\section{Uniqueness of the Infinite Boltzmann Hierarchy}\label{sec:Uniqueness of the Infinite Boltzmann Hierarchy}

The Boltzmann hierarchy for a sequence $\lr{f^{(k)}}_{k=1}^{\infty}$ is an infinite linear system given by
\begin{equation}\label{equ:boltzmann hierarchy}
\left\{
\begin{aligned}
\pa_{t}f^{(k)}+V_{k}\cdot\nabla_{X_{k}}f^{(k)}=&Q^{(k+1)}f^{(k+1)},\\
f^{(k)}(0)=&f_{0}^{(k)},
\end{aligned}
\right.
\end{equation}
where $X_{k}=(x_{1},..,x_{k})$ and $V_{k}=(v_{1},..,v_{k})$.
Here, the collision operator can be decomposed into
\begin{align*}
Q^{(k+1)}f^{(k+1)}=\sum_{j=1}^{k}Q_{j,k+1}f^{(k+1)}
=&\sum_{j=1}^{k}\lrs{Q_{j,k+1}^{+}f^{(k+1)}-Q_{j,k+1}^{-}f^{(k+1)}},
\end{align*}
where
\begin{align*}
Q_{j,k+1}^{+}f^{(k+1)}(X_{k},V_{k})=&\int_{\mathbb{R}^{d}}\int_{\mathbb{S}^{d-1}} f^{(k+1)}(t,X_{k},x_{j},V_{k,j}^{*},v_{k+1}^{*}) B(v_{k+1}-v_{j},\omega)dv_{k+1}d\omega,\\
Q_{j,k+1}^{-}f^{(k+1)}(X_{k},V_{k})=&\int_{\mathbb{R}^{d}}\int_{\mathbb{S}^{d-1}} f^{(k+1)}(t,X_{k},x_{j},V_{k},v_{k+1})  B(v_{k+1}-v_{j},\omega) dv_{k+1} d\omega,
\end{align*}
with the notation $V_{k,j}^{*}=(v_{1},..,v_{j-1},v_{j}^{*},v_{j+1},..,v_{k})$.

Given a solution $f(t)$
 to the Boltzmann equation \eqref{equ:Boltzmann}, it generates a factorized solution to the Boltzmann hierarchy \eqref{equ:boltzmann hierarchy} by letting
\begin{equation}\label{equ:factorized solution}
f^{(k)}(t)=f^{\otimes k}(t).
\end{equation}
Due to the relation \eqref{equ:factorized solution}, the unconditional uniqueness of the Boltzmann hierarchy implies the unconditional uniqueness of the Boltzmann equation. Therefore, we deal with the uniqueness problem by the hierarchy method and prove the following Theorem for the Boltzmann hierarchy.
\begin{theorem}[Uncondition uniqueness for the Boltzmann hierarchy]\label{thm:uniqueness hierarchy}
There is at most one solution $\Ga=\lr{f^{(k)}}_{k\geq 1}$ to the Boltzmann hierarchy \eqref{equ:boltzmann hierarchy} in $[0,T_{0}]$ subject to the conditions that
\begin{enumerate}[$(1)$]
\item $\Ga$ is admissible in the sense that
\begin{align}\label{equ:representation}
f^{(k)}(t)=\int f^{\otimes k}d\nu_{t}(f),
\end{align}
where $d\nu_{t}(f)$ is a probability measure on the space of probability density functions.
\item $\Ga$ satisfies the uniform regularity bounds, that is, there exists a constant $C$ such that
\begin{align}\label{equ:uniform bounds,hierarchy}
\sup_{t\in [0,T_{0}]}\int \n{\lra{\nabla_{x}}^{s}
\lra{v}^{r}f}_{L_{x}^{2}L_{v}^{2}}^{k}d\nu_{t}(f)\leq C^{k}, \quad \f k\geq 1.
\end{align}
\end{enumerate}
where the index $(s,r)$ is the same as in Theorem $\ref{thm:main theorem}$.
\end{theorem}

The admissible condition in Theorem \ref{thm:uniqueness hierarchy} can, in fact, be relaxed by invoking the Hewitt-Savage theorem. Regarding the uniform regularity bounds condition, like in \cite{chen2015unconditional},
under representation \eqref{equ:representation}, we can use Chebyshev's inequality to reformulate \eqref{equ:uniform bounds,hierarchy} into a support condition that $\operatorname{supp}\nu_{t}\subset \mathbb{B}$ where
\begin{align*}
\mathbb{B}=:\lr{f:\n{f}_{H_{x}^{s}L_{v}^{2,r}}\leq C}.
\end{align*}

To match the notations of dispersive and bilinear estimates established in
Sections \ref{sec:Strichartz Estimates}--\ref{sec:Bilinear Estimates},
we might as well work in the Fourier side and set
\begin{align*}
\wt{f}^{(k)}(t)=\mathcal{F}_{(v_{1},..,v_{k})\mapsto (\xi_{1},..,\xi_{k})}^{-1}(f^{(k)}(t)).
\end{align*}
We rewrite $\wt{f}^{(k)}$ in Duhamel form
\begin{align}\label{equ:duhamel,boltzmann hierarchy}
\wt{f}^{(k)}(t)=U^{(k)}\wt{f}^{(k)}(0)+\int_{0}^{t}U^{{(k)}}(t-s)\wt{Q}^{(k+1)}\wt{f}^{(k+1)}(s)ds,
\end{align}
with $U^{(k)}(t)=\prod_{j=1}^{k}e^{it\nabla_{\xi_{j}}\cdot\nabla_{x_{j}}}$.

Let $f_{1}$ and $f_{2}$ be two $C([0,T];H_{x}^{s}L_{v}^{2,r})$ solutions to the Boltzmann equation \eqref{equ:Boltzmann} with the same initial
 datum.
Set
\begin{align*}
 \Gamma_{1}=&\lr{\wt{f}_{1}^{\otimes k}}_{k=1}^{\infty}, \quad
 \Gamma_{2}=\lr{\wt{f}_{2}^{\otimes k}}_{k=1}^{\infty}, \\
\Gamma=&\lr{\wt{f}_{1}^{\otimes k}-\wt{f}_{2}^{\otimes k}}_{k=1}^{\infty}=\lr{\wt{f}^{(k)}(t)=\int \wt{f}^{\otimes k}d\nu_{t}(\wt{f})},
\end{align*}
where $\nu_{t}(\wt{f})=\delta_{\wt{f}_{1}(t)}(\wt{f})-\delta_{\wt{f}_{2}(t)}(\wt{f})$.

To conclude the uniqueness of the Boltzmann equation, it suffices to prove $\Ga=0$.
Since the Boltzmann hierarchy \eqref{equ:duhamel,boltzmann hierarchy} is linear, $\Gamma$ is hence a solution to the Boltzmann
 hierarchy with zero initial datum. Thus after iterating $(\ref{equ:duhamel,boltzmann hierarchy})$ $k$ times, we can write
\begin{align}
\wt{f}^{(1)}(t_{1},x_{1},\xi_{1})=\int_{0}^{t_{1}}\int_{0}^{t_{2}}\ccc \int_{0}^{t_{k}}
J^{(k+1)}(\wt{f}^{(k+1)}(t_{k+1}))d\underline{t}_{k+1}
\end{align}
where $\underline{t}_{k+1}=(t_{2},t_{3},...,t_{k+1})$ and
\begin{align}
J^{(k+1)}(\wt{f}^{(k+1)})(t_{1},\underline{t}_{k+1})=U^{(1)}_{1,2}\wt{Q}^{(2)}U_{2,3}^{(2)}\wt{Q}^{(3)}\ccc
U_{k,k+1}^{(k)}\wt{Q}^{(k+1)}\wt{f}^{(k+1)},
\end{align}
with the notation that $U_{j,j+1}^{(j)}:=U^{(j)}(t_{j}-t_{j+1})$.

As $\wt{Q}^{(k+1)}$ has $k$ terms, there are $k!$ summands inside $\wt{f}(t_{1})$. More precisely,
\begin{align}
\wt{f}(t_{1})=\sum_{\mu}\int_{t_{1}\geq t_{2}\geq \ccc\geq t_{k+1}}J_{\mu}^{(k+1)}(\wt{f}^{(k+1)})(t_{1},\underline{t}_{k+1})
d\underline{t}_{k+1},
\end{align}
where
\begin{align}\label{equ:J,k,formula}
J_{\mu}^{(k+1)}(\wt{f}^{(k+1)})(t_{1},\underline{t}_{k+1})=U_{1,2}^{(1)}\wt{Q}_{\mu(2),2}^{(2)}U_{2,3}^{(2)}\wt{Q}_{\mu(3),3}^{(3)}\ccc
U_{k,k+1}^{(k)}\wt{Q}^{(k+1)}_{\mu(k+1),k+1}\wt{f}^{(k+1)}.
\end{align}
Here, $\lr{\mu}$ is a set of maps from $\lr{2,3,..,k+1}$ to $\lr{1,2,...,k}$ satisfying
that $\mu(2)=1$ and $\mu(j)<j$ for all $j$.

In the quantum setting, there are some known combinatorics
based on Feymann diagrams. Here, we rely on a
combinatorics argument, a Klainerman-Machedon board game, to reduce the number of terms by combining them.

\begin{lemma}[Klainerman-Machedon board game, \cite{KM08}]\label{lemma:KM board game}
One can rewrite
\begin{align}\label{equ:km,f,I}
\wt{f}^{(1)}(t_{1})=\sum_{\mu\in D_{k}}I_{\mu}^{(k+1)}(\wt{f}^{(k+1)})(t_{1}),
\end{align}
where there are at most $4^{k}$ terms inside the set $D_{k}$. Here,
\begin{align}
I_{\mu}^{(k+1)}(\wt{f}^{(k+1)})(t_{1})=\int_{T(\mu)}J_{\mu}^{(k+1)}(\wt{f}^{(k+1)})(t_{1},\underline{t}_{k+1})
d\underline{t}_{k+1},
\end{align}
where the time integration domain $T(\mu)$ is a subset of $[0,t_{1}]^{k}$, depending on $\mu$.
\end{lemma}
For the proof of Lemma \ref{lemma:KM board game}, see \cite[Lemma 6.5]{chen2023derivationboltzmann} in which new techniques are presented and the time integration domain $T(\mu)$ is explicitly calculated.
The set $D_{k}$ serves as a set of representatives for an equivalence class, whereas the index $\mu$ in \eqref{equ:km,f,I} is a specific representative chosen from a distinct equivalence class. This differs from the set $\lr{\mu}$ appearing in \eqref{equ:J,k,formula}, which contains $k!$ elements.

\subsection{Duhamel Expansion and Duhamel Tree}\label{sec:Duhamel Expansion and Duhamel Tree}
We initiate the analysis of the Duhamel expansion by constructing a Duhamel tree, abbreviated as a $D$-tree, and demonstrate the derivation of the Duhamel expansion $J_{\mu}^{(k+1)}$ from this tree. Initially, we introduce an algorithm for generating a $D$-tree based on a collapsing map $\mu$, and illustrate this process with an example. Following this, we establish a general algorithm to compute the Duhamel expansion. See also \cite{chen2022unconditional,CSZ22} for the Duhamel tree diagram representation for the NLS case.

\begin{algorithm}[Duhamel Tree] \label{algorithm:duhamel tree}
~\\

$(1)$ Let $D^{(1)}$ be a starting node in the $D$-tree, and $D^{(2)}$ be the middle child of $D^{(1)}$.
Set counter $j=2$.

$(2)$
Given $j$,
find the pair of indices $l$ and $r$ so that
$l\geq j+1$, $r\geq j+1$ and
\begin{align*}
\mu(l)=\mu(j),
\quad \mu(r)=j,
\end{align*}
and moreover $l$ and $r$ are the minimal indices for which the above equalities hold. Then place
$D^{(l)}$ or $D^{(r)}$ as the left or middle child of $D^{(j)}$ in the $D$-tree. If there is no such $l$ or $r$,
place $F_{\mu(j)}$ or $F_{j}$ as the left or right child of $D^{(j)}$ in the $D$-tree.

$(3)$ If $j=k+1$, then stop, otherwise set $j=j+1$ and go to step $(2)$.

\end{algorithm}

\begin{example}\label{ex:example}
Let us work with the following example.
$$
\begin{tabular}{c|cccc}
$j$&2&3&4&5\\
\hline
$\mu(j)$ &1&1&2&3\\
\end{tabular}
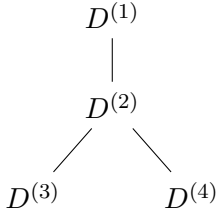
$$
\begin{minipage}{0.3\textwidth}
\centering
\begin{tikzpicture}
\node{$D^{(1)}$}[sibling distance=60pt,level distance=1.2cm]
child{node{$D^{(2)}$}
child{node{$D^{(3)}$}}
child{node{$D^{(4)}$}}
};
\end{tikzpicture}
\end{minipage}
\begin{minipage}{0.67\textwidth}
Let $D^{(1)}$ be a starting node, and $D^{(2)}$ be the middle child of $D^{(1)}$.
Now we start with the counter $j=2$, so we need to find the minimal $l\geq 3$, $r\geq 3$
such that $\mu(l)=1$, and $\mu(r)=2$. In this example, it is $l=3$ and
$r=4$ so we put $D^{(3)}$ as the left child and $D^{(4)}$ as the right child of $D^{(2)}$, in
the $D$-tree as shown in the left.
\end{minipage}

\begin{minipage}{0.67\textwidth}
Next, we move to the counter $j=3$. We find that $\mu(5)=3$ and there is no $l\geq 4$ such that $\mu(l)=\mu(3)=1$. Hence, we put $F_{1}$ as the left child and $D^{(5)}$ as the right child of $D^{(3)}$.
\end{minipage}
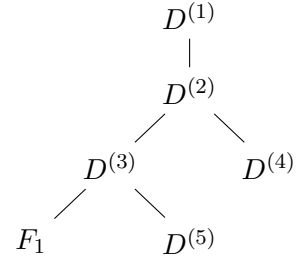
\begin{minipage}{0.3\textwidth}
\centering
\begin{tikzpicture}
\node{$D^{(1)}$}[sibling distance=60pt,level distance=1cm]
child{node{$D^{(2)}$}
child{node{$D^{(3)}$}
child{node{$F_{1}$}}
child{node{$D^{(5)}$}}}
child{node{$D^{(4)}$}}
}
;
\end{tikzpicture}
\end{minipage}

Finally, by iteratively applying the steps outlined in Algorithm \ref{algorithm:duhamel tree}, we construct the complete $D$-tree shown in the following diagram \eqref{figure:duhamel tree}.
\begin{figure}[htpb]
\caption{Duhamel Tree}
\label{figure:duhamel tree}
\begin{tikzpicture}
\node{$D^{(1)}$}[sibling distance=80pt,level distance=1cm]
child{node{$D^{(2)}$}
child{node{$D^{(3)}$}[sibling distance=40pt,level distance=1cm]
child{node{$F_{1}$}}
child{node{$D^{(5)}$}
child{node{$F_{3}$}}
child{node{$F_{5}$}}}}
child{node{$D^{(4)}$}[sibling distance=40pt,level distance=1cm]
child{node{$F_{2}$}}
child{node{$F_{4}$}}}
}
;
\end{tikzpicture}
\end{figure}

Now, we use the $D$-tree \eqref{figure:duhamel tree} to generate the Duhamel
 expansion $J_{\mu}^{(5)}(\wt{f}^{\otimes 5})$. On the one hand, by \eqref{equ:J,k,formula}, one can directly obtain
\begin{align}
J_{\mu}^{(5)}(\wt{f}^{\otimes 5})(t_{1},\underline{t}_{5})=
U_{1,2}^{(1)}\wt{Q}_{1,2}^{(2)}U_{2,3}^{(2)}\wt{Q}_{1,3}^{(3)}U_{3,4}^{(3)}\wt{Q}_{2,4}^{(4)}
U_{4,5}^{(4)}\wt{Q}^{(5)}_{3,5}\wt{f}^{\otimes 5}.
\end{align}
According to the $D$-tree \eqref{figure:duhamel tree}, we define
\begin{align*}
D^{(1)}=&U_{1}D^{(2)},\\
D^{(2)}=&U_{-2}(\wt{Q}(U_{2}D^{(3)},U_{2}D^{(4)})),\\
D^{(3)}=&U_{-3}(\wt{Q}(U_{3,5}\wt{f},U_{3}D^{(5)})),\\
D^{(4)}=&U_{-4}(\wt{Q}(U_{4,5}\wt{f},U_{4,5}\wt{f})),\\
D^{(5)}=&U_{-5}(\wt{Q}(\wt{f},\wt{f})),
\end{align*}
with the notation that $U_{-i}=U(-t_{i})$ and $U_{i,j}=U(t_{i}-t_{j})$.
Then by expanding $J_{\mu}^{(5)}(\wt{f}^{\otimes 5})$ and $D^{(1)}$, we actually have
\begin{align}
J_{\mu}^{(5)}(\wt{f}^{\otimes 5})=D^{(1)}.
\end{align}

\end{example}

Next, we present a more general algorithm for obtaining the Duhamel expansion from the $D$-tree.
\begin{algorithm}[From $D$-tree to Duhamel expansion]\label{algorithm:from d-tree to duhamel expansion}
~\\
\hspace*{1em}$(1)$ Set $F_{i}=U_{-k-1}\wt{f}$ for $i=1,2,...,k+1$, and
\begin{align*}
&D^{(k+1)}=U_{-k-1}(\wt{Q}(\wt{f},\wt{f})).
\end{align*}
 Set counter $j=k$.

$(2)$ Given $j$, set
\begin{align*}
&D^{(j)}=U_{-j}(\wt{Q}(U_{j}C_{r},U_{j}C_{l})),
\end{align*}
where $C_{l}$ is the left child and $C_{r}$ is the right child of $D^{(l)}$ in the $D$-tree.

$(3)$ Set $j=j-1$. If $j=1$, set
\begin{align*}
&D^{(1)}=U_{1}D^{(2)},
\end{align*}
and stop, otherwise go to step $(2)$.

\end{algorithm}

By employing Algorithm \ref{algorithm:from d-tree to duhamel expansion}, we derive the following proposition.
\begin{proposition}\label{prop:from d-tree to duhamel expansion}
There holds that
\begin{align}
J_{\mu}^{(k+1)}(\wt{f}^{\otimes (k+1)})(t_{1},\underline{t}_{k+1})=D^{(1)}(t_{1},\underline{t}_{k+1}).
\end{align}
\end{proposition}

\subsection{Iterative Estimates}\label{sec:Iterative Estimates}
In the section, we delve into the estimates of the Duhamel expansions.
We first provide a time-independent estimate, which is used to deal with the roughest term $D^{(k+1)}$.
\begin{lemma}\label{lemma:rough term}
Let $s\geq \frac{d}{4}$ and $r>\frac{d}{2}+\ga>0$. Then we have
\begin{align}\label{equ:rough term}
\n{\wt{Q}^{\pm}(\wt{f},\wt{g})}_{L_{x}^{2}H_{\xi}^{r}}\lesssim \n{\wt{f}}_{H_{x}^{s}H_{\xi}^{r}}\n{\wt{g}}_{H_{x}^{s}H_{\xi}^{r}}
\end{align}
\end{lemma}
\begin{proof}
For the loss term, we use Lemma \ref{lemma:Q-,bilinear estimate} with $p=2$, H\"{o}lder inequality, and Sobolev inequality to get
\begin{align}
\n{\wt{Q}^{-}(\wt{f},\wt{g})}_{L_{x}^{2}H_{\xi}^{r}}\lesssim \n{\lra{\nabla_{\xi}}^{r}\wt{f}}_{L_{x}^{4}L_{\xi}^{2}}\n{\wt{g}}_{L_{x}^{4}H_{\xi}^{r}}\lesssim \n{\wt{f}}_{H_{x}^{s}H_{\xi}^{r}}\n{\wt{g}}_{H_{x}^{s}H_{\xi}^{r}}
\end{align}
provided that $s\geq \frac{d}{4}$ and $r>\frac{d}{2}+\ga>0$.

For the gain term, by duality,
\eqref{equ:rough term} is equivalent to
\begin{align}\label{equ:Q+,rough term,equivalent}
\int \wt{Q}^{+}(\wt{f},\wt{g}) h dx d\xi \lesssim
\n{\wt{f}}_{H_{x}^{s}H_{\xi}^{r}}\n{\wt{g}}_{H_{x}^{s}H_{\xi}^{r}} \n{h}_{L_{x}^{2}H_{\xi}^{-r}} .
\end{align}
Following the frequency analysis as in Section \ref{sec:Bilinear Estimates for the Gain Term}, we
insert a Littlewood-Paley decomposition to get
\begin{align*}
\eqref{equ:Q+,rough term,equivalent}=\sum_{M_{1},M_{2},M_{3}} \int \wt{Q}^{+}(P_{M_{1}}^{\xi}\wt{f},P_{M_{2}}^{\xi}\wt{g})P_{M_{3}}^{\xi}h dx d\xi.
\end{align*}
It suffices to consider Case $A$: $M_{1}\geq M_{2}$, $M_{1}\geq M_{3}$, as other cases follow similarly.
Let $I_{A}$ denote the integral restricted to the Case $A$.

By Cauchy-Schwarz inequality and estimate \eqref{equ:Q+,bilinear estimate,PM,f,g} in Lemma \ref{lemma:Q+,bilinear estimate}, we have
\begin{align*}
I_{A}\lesssim&
\sum_{ M_{1}\geq M_{2},M_{1}\geq M_{3}}
 \n{\wt{Q}^{+}(P_{M_{1}}^{\xi}\wt{f},P_{M_{2}}^{\xi}\wt{g})}_{L_{x,\xi}^{2}}
\n{P_{M_3}^{\xi}h}_{L_{x,\xi}^{2}}\\
\lesssim&\sum_{ M_{1}\geq M_{2},M_{1}\geq M_{3}}
 \n{P_{M_{1}}^{\xi}\wt{f}}_{L_{x}^{4}L_{\xi}^{2}}\n{P_{M_{2}}^{\xi}\wt{g}}_{L_{x}^{4}W_{\xi}^{r_{p},p}}
\n{P_{M_3}^{\xi}h}_{L_{x,\xi}^{2}}.
\end{align*}
Choosing $p=2$ and carrying out the summation in $M_{2}$, we get
\begin{align*}
I_{A}\lesssim \sum_{M_{1}\geq M_{3}}
 \n{P_{M_{1}}^{\xi}\wt{f}}_{L_{x}^{4}L_{\xi}^{2}}\n{\wt{g}}_{L_{x}^{4}H_{\xi}^{r}}
\n{P_{M_3}^{\xi}h}_{L_{x,\xi}^{2}}.
\end{align*}
By Bernstein inequality, Cauchy-Schwarz inequality, Minkowski inequality, and Sobolev inequality, we arrive at
\begin{align*}
I_{A}\lesssim& \sum_{M_{1}\geq M_{3}} \frac{M_{3}^{r}}{M_{1}^{r}}
 \n{P_{M_{1}}^{\xi}\wt{f}}_{L_{x}^{4}H_{\xi}^{r}}\n{\wt{g}}_{L_{x}^{4}H_{\xi}^{r}}
\n{P_{M_3}^{\xi}h}_{L_{x}^{2}H_{\xi}^{-r}}\\
\lesssim& \n{\wt{f}}_{L_{x}^{4}H_{\xi}^{r}}\n{\wt{g}}_{L_{x}^{4}H_{\xi}^{r}}\n{h}_{L_{x}^{2}H_{\xi}^{-r}}\\
\lesssim& \n{\wt{f}}_{H_{x}^{s}H_{\xi}^{r}}\n{\wt{g}}_{H_{x}^{s}H_{\xi}^{r}} \n{h}_{L_{x}^{2}H_{\xi}^{-r}},
\end{align*}
which completes the proof of \eqref{equ:Q+,rough term,equivalent}.
\end{proof}

Now, we get into the analysis of iterative estimates.
\begin{lemma}\label{lemma:iterative estimate}
Let $t_{1}\in [0,T]$ and $\wt{f}^{(k)}(t)=\int_{\mathbb{B}} \wt{f}^{\otimes k}d\nu_{t}(\wt{f})$. Then we have
\begin{align}\label{equ:iterative estimates}
\n{I_{\mu}^{(k+1)}(\wt{f}^{(k+1)}(t_{1}))}_{L_{T}^{\infty}L_{x_{1},\xi_{1}}^{2}}^{2}\leq (CT^{\frac{1}{2}})^{k-1}\int_{0}^{T}\int_{\mathbb{B}} \n{\wt{f}}_{L_{T}^{\infty}H_{x}^{s}H_{\xi}^{r}}^{k+1}d|\nu_{t_{k+1}}|(\wt{f})dt_{k+1}.
\end{align}
\end{lemma}
\begin{proof}
We begin by considering Example \ref{ex:example}. We shall demonstrate how to iteratively apply the space-time bilinear estimates which were established in Section \ref{sec:Bilinear Estimates}, to arrive at \eqref{equ:iterative estimates}.
Noticing that
\begin{align*}
I_{\mu}^{(5)}(\wt{f}^{(5)})(t_{1})=\int_{T(\mu)}\int J_{\mu}^{(5)}(\wt{f}^{\otimes 5})(t_{1},\underline{t}_{5})
d\nu_{t_{5}}(\wt{f})d\underline{t}_{5},
\end{align*}
we put $I_{\mu}^{(k+1)}(\wt{f}^{(k+1)})$ in the $L_{x}^{2}L_{\xi}^{2}$ norm and apply Minkowski inequality to get
\begin{align*}
I:=\n{I_{\mu}^{(5)}(\wt{f}^{(5)})(t_{1})}_{L_{x_{1},\xi_{1}}^{2}}\leq &
\int_{[0,T]^{4}}\int \n{J_{\mu}^{(5)}(\wt{f}^{\otimes 5})(t_{1},\underline{t}_{5})}_{L_{x,\xi}^{2}}
d|\nu_{t_{5}}|(\wt{f})d\underline{t}_{5}\\
=&\int_{[0,T]^{4}}\int \n{D^{(1)}(t_{1},\underline{t}_{5})}_{L_{x,\xi}^{2}}
d|\nu_{t_{5}}|(\wt{f})d\underline{t}_{5}\\
=&\int_{[0,T]^{4}}\int \n{D^{(2)}(\underline{t}_{5})}_{L_{x,\xi}^{2}}
d|\nu_{t_{5}}|(\wt{f})d\underline{t}_{5},
\end{align*}
where in the last equality we have used that
 $D^{(1)}(t_{1},\underline{t}_{5})=U(t_{1})D^{(2)}(\underline{t}_{5})$ and the unitary property of $U(t)$.

On the one hand, we have that $D^{(2)}=U_{-2}(\wt{Q}(U_{2}D^{(3)},U_{2}D^{(4)}))$. On the other hand, we note that, from the Duhamel tree
diagram \ref{figure:duhamel tree}, the roughest term $D^{(5)}$ is the offspring of $D^{(3)}$. Therefore, carrying out the
$dt_{2}$ integral, we apply the space-time bilinear estimate \eqref{equ:Q,bilinear estimate,Hs,1} with $s_{1}=0$ to reach
\begin{align*}
&\int_{[0,T]^{4}}\int \n{D^{(2)}(\underline{t}_{5})}_{L_{x,\xi}^{2}}
d|\nu_{t_{5}}|(\wt{f})d\underline{t}_{5}\notag\\
\leq& CT^{\frac{1}{2}}\int_{\mathbb{B}} \int_{[0,T]^{3}} \n{D^{(3)}(t_{3},t_{5})}_{L_{x}^{2}H_{\xi}^{r}}\n{D^{(4)}(t_{4},t_{5})}_{H_{x}^{s}H_{\xi}^{r}}
d|\nu_{t_{5}}|(\wt{f})dt_{3}dt_{4}dt_{5}.
\end{align*}

Due to that $D^{(3)}(t_{3},t_{5})=U_{-3}(\wt{Q}(U_{3,5}\wt{f},U_{3}D^{(5)}))$, carrying out the $dt_{3}$ integral, we apply the space-time bilinear estimate \eqref{equ:Q,bilinear estimate,Hs,2} with $s_{1}=0$ to get
\begin{align*}
I\leq& (CT^{\frac{1}{2}})^{2}\int_{\mathbb{B}} \int_{[0,T]^{2}} \n{U_{-5}\wt{f}}_{H_{x}^{s}H_{\xi}^{r}}\n{D^{(5)}(t_{5})}_{L_{x}^{2}H_{\xi}^{r}}\n{D^{(4)}(t_{4},t_{5})}_{H_{x}^{s}H_{\xi}^{r}}
d|\nu_{t_{5}}|(\wt{f})dt_{4}dt_{5}.
\end{align*}

Due to that $D^{(4)}(t_{4},t_{5})=U_{-4}(\wt{Q}(U_{4,5}\wt{f},U_{4,5}\wt{f}))$, carrying out the $dt_{4}$ integral, we apply the space-time bilinear estimate \eqref{equ:Q,bilinear estimate,Hs,2} with $s_{1}=s$ to get
\begin{align*}
I\leq& (CT^{\frac{1}{2}})^{3}\int_{\mathbb{B}} \int_{[0,T]} \n{U_{-5}\wt{f}}_{H_{x}^{s}H_{\xi}^{r}}^{3}\n{D^{(5)}(t_{5})}_{L_{x}^{2}H_{\xi}^{r}}
d|\nu_{t_{5}}|(\wt{f})dt_{5}.
\end{align*}
Finally, applying the time-independent bilinear estimate \eqref{equ:rough term} in Lemma \ref{lemma:rough term}, we arrive at
\begin{align*}
I\leq (CT^{\frac{1}{2}})^{3}\int_{\mathbb{B}} \int_{[0,T]}  \n{\wt{f}}_{L_{T}^{\infty}H_{x}^{s}H_{\xi}^{r}}^{5}d|\nu_{t_{5}}|(\wt{f})dt_{5},
\end{align*}
which yields the desired estimate \eqref{equ:iterative estimates} for this example.

We can now present the algorithm to prove the general case.
\begin{algorithm}

We first write out
\begin{align}
I_{\mu}^{(k+1)}(\wt{f}^{(k+1)})(t_{1})=\int_{T(\mu)}\int_{\mathbb{B}} J_{\mu}^{(k+1)}(\wt{f}^{\otimes (k+1)})(t_{1},\underline{t}_{k+1})
d\nu_{t_{k+1}}(\wt{f})d\underline{t}_{k+1}.
\end{align}

Step 1. Put $I_{\mu}^{(k+1)}(\wt{f}^{(k+1)})$ in the $L_{x}^{2}L_{\xi}^{2}$ norm. By Minkowski inequality and Proposition \ref{prop:from d-tree to duhamel expansion}, we have
\begin{align*}
&\n{I_{\mu}^{(k+1)}(\wt{f}^{(k+1)}(t_{1}))}_{L_{x_{1},\xi_{1}}^{2}}^{2}\\
\leq &
\int_{[0,T]^{k}}\int_{\mathbb{B}}  \n{J_{\mu}^{(k+1)}(\wt{f}^{\otimes (k+1)})(t_{1},\underline{t}_{k+1})}_{L_{x,\xi}^{2}}
d|\nu_{t_{k+1}}|(\wt{f})d\underline{t}_{k+1}\\
=&\int_{[0,T]^{k}}\int_{\mathbb{B}}  \n{D^{(1)}(t_{1},\underline{t}_{k+1})}_{L_{x,\xi}^{2}}
d|\nu_{t_{k+1}}|(\wt{f})d\underline{t}_{k+1}\\
=&\int_{[0,T]^{k}}\int_{\mathbb{B}}  \n{D^{(2)}(\underline{t}_{k+1})}_{L_{x,\xi}^{2}}
d|\nu_{t_{k+1}}|(\wt{f})d\underline{t}_{k+1}
\end{align*}
where in the last equality we have used that $D^{(1)}=U_{1}D^{(2)}$.
Set the counter $j=2$.

Step 2. Noticing that $D^{(j)}=U_{-j}\wt{Q}(U_{j}C_{l},U_{j}C_{r})$,
we use the bilinear estimates, according to the position of the roughest term $D^{(k+1)}$ as shown in the Duhamel tree.
\begin{enumerate}[$(i)$]
\item
If $D^{(k+1)}$ is the offspring of the left branch of $D^{(j)}$, we apply the bilinear estimate \eqref{equ:Q,bilinear estimate,Hs,1} with
 $s_{1}=0$ to obtain
\begin{align*}
\int_{[0,T]}\n{D^{(j)}}_{L_{x}^{2}H_{\xi}^{r}}dt_{j}\leq CT^{\frac{1}{2}}\n{C_{l}}_{L_{x}^{2}H_{\xi}^{r}}
\n{C_{r}}_{H_{x}^{s}H_{\xi}^{r}}.
\end{align*}
\item
If $D^{(k+1)}$ is the offspring of the right branch of $D^{(j)}$, we apply the bilinear estimate \eqref{equ:Q,bilinear estimate,Hs,2}
with $s_{1}=0$ to obtain
\begin{align*}
\int_{[0,T]}\n{D^{(j)}}_{L_{x}^{2}H_{\xi}^{r}}dt_{j}\leq CT^{\frac{1}{2}}
\n{C_{l}}_{H_{x}^{s}H_{\xi}^{r}}\n{C_{r}}_{L_{x}^{2}H_{\xi}^{r}}.
\end{align*}
\item
If $D^{(k+1)}$ is neither the left nor the right offspring of $D^{(j)}$, we apply the bilinear estimate \eqref{equ:Q,bilinear estimate,Hs,1} or
\eqref{equ:Q,bilinear estimate,Hs,2} with $s_{1}=s$ to obtain
\begin{align*}
\int_{[0,T]}\n{D^{(j)}}_{H_{x}^{s}H_{\xi}^{r}}dt_{j}\leq CT^{\frac{1}{2}}
\n{C_{l}}_{H_{x}^{s}H_{\xi}^{r}}\n{C_{r}}_{H_{x}^{s}H_{\xi}^{r}}.
\end{align*}
\end{enumerate}
\end{algorithm}

Step 3. Set $j=j+1$. If $j\leq k$, go to Step 2. If $j=k+1$, we apply time-independent bilinear estimates \eqref{equ:rough term} to deal with $D^{(k+1)}$ and obtain
\begin{align*}
\n{D^{(k+1)}}_{L_{x}^{2}H_{\xi}^{r}}=\n{U(t_{k+1})\wt{Q}(\wt{f},\wt{f})}_{L_{x}^{2}H_{\xi}^{r}}
=\n{\wt{Q}(\wt{f},\wt{f})}_{L_{x}^{2}H_{\xi}^{r}}\leq C\n{\wt{f}}_{L_{T}^{\infty}H_{x}^{s}H_{\xi}^{r}}^{2}.
\end{align*}
Now, we have completed the iterative estimate part, in which the space-time bilinear estimates are used $(k-1)$
 times in total, and hence arrive at \eqref{equ:iterative estimates}.
\end{proof}

\subsection{Proof of the Main Theorem}\label{sec:Proof of the Main Theorem}

\begin{proof}[\textbf{Proof of Theorems $\ref{thm:main theorem}$ and $\ref{thm:uniqueness hierarchy}$}]
We focus on the proof of Theorem \ref{thm:main theorem}, as it also applies to Theorem $\ref{thm:uniqueness hierarchy}$.
Let $f_{1}$ and $f_{2}$ be two $C([0,T];H_{x}^{s}L_{v}^{2,r})$
solutions to the Boltzmann equation \eqref{equ:Boltzmann} with the same initial datum $f_{0}$. We work in the
Fourier side, and set
$$\wt{f}(t)=\wt{f}_{1}(t)-\wt{f}_{2}(t),$$
which yields
\begin{align}
\wt{f}^{(k)}(t)=\int \wt{f}^{\otimes k}d\nu_{t}(\wt{f}),
\end{align}
where $\nu_{t}(\wt{f})=\delta_{\wt{f}_{1}(t)}(\wt{f})-\delta_{\wt{f}_{2}(t)}(\wt{f})$. Then by Lemma \ref{lemma:KM board game}, we rewrite
\begin{align*}
\wt{f}(t_{1})=\sum_{\mu\in D_{k}}I_{\mu}^{(k+1)}(\wt{f}^{(k+1)})(t_{1}),
\end{align*}
where there are at most $4^{k}$ terms inside the set $D_{k}$. By iterative estimates in Lemma \ref{lemma:iterative estimate}, we have
\begin{align*}
\n{\wt{f}(t_{1})}_{L_{T}^{\infty}L_{x,\xi}^{2}}\leq& \sum_{\mu\in D_{k}}\n{I_{\mu}^{(k+1)}(\wt{f}^{(k+1)})}_{L_{T}^{\infty}L_{x,\xi}^{2}}\\
\leq& \sum_{\mu\in D_{k}}(CT^{\frac{1}{2}})^{k-1}\int_{0}^{T}\int_{\mathbb{B}} \n{\wt{f}}_{L_{t}^{\infty}H_{x}^{s}H_{\xi}^{r}}^{k+1}d|\nu_{t_{k+1}}|(\wt{f})dt_{k+1}\\
\leq& \lrs{4CC_{0}T^{\frac{1}{2}}}^{k},
\end{align*}
where $C_{0}=\max\lr{\n{\wt{f}_{1}}_{L_{T_{0}}^{\infty}H_{x}^{s}H_{\xi}^{r}},\n{\wt{f}_{2}}_{L_{T_{0}}^{\infty}H_{x}^{s}H_{\xi}^{r}}}$.
Choosing $T$ small enough such that $4CC_{0}T^{\frac{1}{2}}\leq \frac{1}{2}$, we have
\begin{align}
\n{\wt{f}(t_{1})}_{L_{T}^{\infty}L_{x,\xi}^{2}}\leq \lrs{\frac{1}{2}}^{k}\to 0, \quad \text{as $k\to \infty$.}
\end{align}
Therefore, we have concluded that $\wt{f}(t)=0$ for $t\in [0,T]$. By a bootstrap argument, we can then fill the whole $[0,T_{0}]$ interval.
\end{proof}

\noindent \textbf{Acknowledgements}
We would like to express our sincere gratitude to the referee for the careful checking of the manuscript, the insightful comments and helpful suggestions.
X. Chen was supported in part by NSF DMS-2406620,
S. Shen was supported in part by the National Key R\&D Program of China under Grant 2024YFA1015500, NSF of China under Grant 12501322, and Anhui Provincial NSF 2508085QA001.
Z. Zhang was supported in part by National Key R\&D Program of China under Grant 2023YFA1008801 and  NSF of China under Grant 12288101.

\noindent\textbf{Data Availability Statement}
Data sharing is not applicable to this article as no datasets were generated or analysed during the current study.

\noindent\textbf{Conflict of Interest}
The authors declare that they have no conflict of interest.


\end{document}